\documentclass[reqno]{amsart}

\usepackage[utf8]{inputenc}
\usepackage{amsmath}
\numberwithin{equation}{section}
\numberwithin{figure}{section}
\numberwithin{table}{section}
\usepackage{graphicx}
\usepackage{amssymb}
\usepackage{mathrsfs}
\usepackage{amsfonts}
\usepackage[colorlinks=true, allcolors=blue]{hyperref}
\usepackage{subfigure} 
\usepackage{color}
\usepackage{extarrows}
\usepackage{booktabs}
\usepackage{hyperref}
\usepackage{tikz}
\usepackage{extarrows}
\usepackage{booktabs}
\usepackage{amsthm}
\usepackage{enumerate}
\usepackage{enumitem} 
\usepackage{indentfirst}
\usepackage{embrac}
\usepackage{cite} 
\usetikzlibrary{positioning,shapes}

\def\bR{{\mathbb R}}

\def\sE{{\mathscr E}}
\def\sF{{\mathscr F}}

\def\bs{\mathbf{s}}

\def\bN{\mathbb{N}}

\def\fm{\mathfrak{m}}

\def\${|\!|\!|}
\def\l|{\left|\!\left|\!\left|}
\def\r|{\right|\!\right|\!\right|}

\newtheorem{theorem}{Theorem}[section]
\newtheorem{lemma}[theorem]{Lemma}

\newtheorem{corollary}[theorem]{Corollary}
\theoremstyle{definition}
\newtheorem{definition}[theorem]{Definition}

\theoremstyle{remark}
\newtheorem{remark}[theorem]{Remark}
\numberwithin{equation}{section}

\begin{document}

\title[Generalization of quasidiffusions]{On generalization of quasidiffusions}

\author{Liping Li}
\address{Fudan University, Shanghai, China.  }
\address{Bielefeld University,  Bielefeld, Germany.}
%\curraddr{Bielefeld University,  Bielefeld, Germany.}
\email{liliping@fudan.edu.cn}
\thanks{The author is a member of LMNS,  Fudan University.  He is also partially supported by NSFC (No.  11931004) and Alexander von Humboldt Foundation in Germany. }

%\author{Hanlai Lin}
%\address{Fudan University, Shanghai, China.  }
%\email{hanlailin18@163.com}

%    General info

%\date{January 1, 2001 and, in revised form, June 22, 2001.}



\begin{abstract}
A quasidiffusion is by definition a time-changed Brownian motion on certain closed subset of $\bR$.  The aim of this paper is two-fold.  On one hand,  we will put forward a generation of quasidiffusion,  called skip-free Hunt process,  by way of a pathwise setup.  As an analogue of regular diffusion on an interval,  a skip-free Hunt process also admits a so-called scale function and a so-called speed measure.  In addition,  it is uniquely determined by scale function and speed measure.  Particularly,  a skip-free Hunt process is a quasidiffusion,  if and only if it is on the natural scale.  On the other hand,  we will give an analytic characterization of skip-free Hunt processes by means of Dirichlet forms.  The main result shows that a so-called strongly local-like property for associated Dirichlet forms is equivalent to the skip-free property of sample paths of skip-free Hunt processes.  As a byproduct,  the explicit expression of Dirichlet forms associated to skip-free Hunt processes will be formulated.  

\end{abstract}

\maketitle
\tableofcontents

\section{Introduction}

As one of the most important stochastic models,  a one-dimensional diffusion process means a continuous strong Markov process $X=(X_t)_{t\geq 0}$ on an interval $I=\langle l, r\rangle$ where $l$ or $r$ may or may not be contained in $I$; see,  e.g., \cite{I06,  IM74,  M68}.  It is called \emph{regular} if $\mathbf{P}_x(T_y<\infty)>0$ for any $x\in \mathring{I}:=(l,r)$ and $y\in I$,  where $T_y:=\inf\{t>0: X_t=y\}$,  and \emph{with no killing inside} if $X_{\zeta-}\notin I$ for $\zeta<\infty$,  where $\zeta$ is the lifetime of $X$.  A significant characterization tells us that a regular diffusion process $X$ with no killing inside is uniquely determined by a \emph{canonical scale function} $\bs$ and a \emph{canonical speed measure} $\fm$; see,  e.g.,  \cite[V\S7]{RW87} and \cite[VII\S3]{RY99}.  Note that $\bs$ is a continuous and strictly increasing function on $I$ giving the hitting distributions of $X$ in the sense that for any $a,x,b\in I$ with $a<x<b$,
\begin{equation}\label{eq:02}
	\mathbf{P}_x(T_b<T_a)=\frac{\bs(x)-\bs(a)}{\bs(b)-\bs(a)}. 
\end{equation}
When $\bs(x)=x$,  $X$ is called \emph{on its natural scale}. 
The canonical speed measure $\fm$ is a fully supported Radon measure on $I$,  which is roughly defined as $-\frac{1}{2}h''_{a,b}$ on every open interval $(a,b)\subset I$,  where $h_{a,b}(x):=\mathbf{E}_xT_a\wedge T_b$ is concave and $-h''_{a,b}$ is the Radon measure induced by the second order derivative of $-h_{a,b}$ in the sense of distribution.  A more comprehensible explanation for canonical speed measure is as follows. When on its natural scale,  $X$ can be expressed as a time change of Brownian motion and $\fm$ measures the speed of its movements: In regions where $\fm$ is large,  $X$ moves slowly.  Particularly,  the Brownian motion is on its natural scale and its canonical speed measure is the Lebesgue measure.  

A famous analytic treatment to reach a regular diffusion process concerns the \emph{generalized second order differential operator}
\begin{equation}\label{eq:01}
\mathscr L:=\frac{1}{2}\frac{d^2}{d\fm d\bs}
\end{equation}
 and its associated \emph{Feller semigroup}.  This semigroup leads to a \emph{Feller process} identifying with the expecting diffusion.  A systematic introduction is referred to in \cite{M68}.  To our knowledge,  it was Kac and Krein \cite{KK74},  who first initialized a spectral theory,  known as \emph{Krein's correspondence},  to generalize the operator \eqref{eq:01} to the case that $\bs(x)=x$ and $\fm$ is determined by a right continuous and (not strictly) increasing function.  This theory applied to Markov process in \cite{K75} and led to so-called \emph{quasidiffusion}.  A quasidiffusion,  also termed as \emph{generalized diffusion} in,  e.g.,  \cite{W74,  KW82, LM20} and \emph{gap diffusion} in,  e.g.,  \cite{K81},  is by definition a \emph{time change} of Brownian motion.  It is quite similar to a regular diffusion,  besides that the continuity of sample paths may fail,  and instead the so-called \emph{skip-free} property holds: A quasidiffusion jumps across only the gaps of $E_m$.  At a heuristic level,  a quasidiffusion may be thought of as the \emph{trace} of Brownian motion on the topological support of $\fm$,  which is also called the \emph{speed measure} of quasidiffusion.  
 
Henceforth there appeared rich investigations of quasidiffuions.  For example,  J\"urgen \cite{G75} studied the operator \eqref{eq:01} of quasidiffusions in a Feller's framework.  Kasahara \cite{K75} raised a semigroup approach to quasidiffusions with the help of Krein's correspondence,  and the semigroup property of quasidiffuions was further discovered in,  e.g.,  \cite{K80}.  Krein's correspondence is actually a resolvent approach to quasidiffuions.  A related consideration also appeared in \cite{LM20}.  Some ingredients in the context of Markov processes for quasidiffusions were examined in,  e.g.,  \cite{K81,  K86,  K87} and the references thereof.  In the context of stochastic analysis,  a quasidiffusion turns to be a semimartingale.  Its semimartingale decomposition was accomplished in \cite{BK87}.  
 
Nevertheless,  as far as we know,  there has not appeared a pathwise setup for quasidiffusion,  as an analogue of regular diffusion.  (Actually time change of Brownian motion is a pathwise definition,  while it is not led to by certain straightforward pathwise properties like continuity,  regularity,  and etc.) This setup becomes our first main object of current paper, because the theory of quasidiffusions would not be complete without including it.  To accomplish it,  we will introduce a so-called \emph{skip-free Hunt process} (SFH in abbreviation) in one dimension.  The state space of SFH is a \emph{nearly closed subset} $E$ of extended real number system $\overline{\bR}$,  an extension of interval $I$,  i.e.  $\overline E:= E\cup \{l,r\}$ is a closed subset of $\overline{\bR}$ where $l=\inf\{x: x\in E\}$ and $r=\sup\{x: x\in E\}$.  Then an SFH is by definition a Hunt process on $E$ with the \emph{regular property} (SR) (similar to that for regular diffusion),  the \emph{skip-free property} (SF) (in place of the continuity of sample paths for regular diffusion) and no killing inside (SK); see Definition~\ref{DEF51} for more details.  This concept obviously generalizes regular diffusion: When $E$ is an interval,  (SF) is identified with the continuity of sample paths and an SFH is actually a regular diffusion.  Certainly a quasidiffusion is an SFH.  (To be rigorous,  the condition (QK) for a quasidiffusion stated in \S\ref{SEC61} should be assumed for this fact.  Note that (QK) bars the possibility of killing at endpoints $l$ and $r$.  In other words,  the \emph{Robin boundary condition} for \eqref{eq:01} will not be under consideration for simplification.) On the contrary,  one of the main results in this paper,  Theorem~\ref{THM71},  shows that every SFH is a quasidiffusion up to a homeomorphism.  More precisely,  it turns out in Theorem~\ref{THM58} that every SFH $X$ gives a strictly increasing,  continuous real valued function $\bs$ on $E$,  called \emph{scale function}, in the same sense as \eqref{eq:02},  and $\tilde{X}_t:=\bs(X_t), t\geq 0$,  an SFH on its natural scale,  is also a quasidiffusion on $\tilde{E}:=\bs(E)$.   Furthermore  when on its natural scale,  $X$ admits a \emph{speed measure} roughly defined as $-\frac{1}{2}h''_{a,b}$ on $E\cap (a,b)$ for any $a,b\in E$ with $a<b$,  where $h_{a,b}(x):=\mathbf{E}_xT_a\wedge T_b, x\in E\cap (a,b)$ can be extended to a concave function on $(a,b)$.  This speed measure is identified with that of $X$ as a quasidiffusion. In a word,  the following processes with identifying speed measures are equivalent: 
\begin{itemize}
\item[(1)] An SFH on its natural scale;  
\item[(2)] A quasidiffusion with no killing inside.  
\end{itemize}
As a corollary,  an SFH is uniquely determined by its scale function and speed measure (up to an affine transformation),  as is analogous to a regular diffusion; see Corollary~\ref{COR64}.  

The second object of this paper is to complete an analytic characterization of SFHs by means of \emph{Dirichlet forms.}  A Dirichlet form is a closed symmetric form with Markovian property on an $L^2$-space.  Theory of Dirichlet forms is very useful for studying symmetric Markov processes,  because every \emph{regular} Dirichlet form is associated to a Hunt process due to the celebrated works  \cite{F71, F71-2} by Fukushima.  We refer readers to \cite{FOT11, CF12} for notations and terminologies in the theory of Dirichlet forms.  As a time-changed Brownian motion,  a quasidiffusion is always symmetric with respect to its speed measure.  The formulation of its associated Dirichlet form then falls into a general framework,  termed as \emph{trace Dirichlet forms} corresponding to \emph{time change transformation} of general Markov processes; see,  e.g.,  \cite{FHY04,  CFY06} as well as \cite[Chapter 5]{CF12}.  The explicit expression of Dirichlet form associated to a quasidiffusion has been already obtained in,  e.g.,  \cite{KW09} and \cite[Theorem~3.1]{L23}.  But the problem we are concerned with is,  like the strongly local property for the Dirichlet form of a regular diffusion,  which property of a Dirichlet form may lead to a quasidiffusion (or a general SFH).  In Definition~\ref{DEF51-2} we will put forward a concept of \emph{strongly local-like} property for a Dirichlet form,  which is quite similar to but slightly more general than the strongly local property.  Then another main result,  Theorem~\ref{THM53},  obtains the equivalence between the strongly local-like property and its association to an SFH for a regular (and irreducible) Dirichlet form.  In other words,  in  comparison with that strongly local property corresponds to the continuity of sample paths,  strongly local-like property leads to the skip-free property of sample paths.  As a byproduct,  the explicit expression of Dirichlet form associated to an SFH will be also formulated in Theorem~\ref{THM53}.

Throughout this paper,  the condition (SK) for an SFH and the condition (QK) for a quasidiffusion are assumed to bar the possibility of killing inside.  We wish to state emphatically that without assuming (QK),  the investigations of quasidiffusion still hold true,  while there may appear killing at endpoints contained in the state space; see Remark \ref{RM73}.  The case of SFH is more complicated.  An SFH without (SK) may admit killing everywhere,  so a \emph{killing measure} should be introduced to characterize the killing part of this process.  We hope to make it in a future contribution. 

The paper is organized as follows.   In \S\ref{SEC5} we will investigate the SFH in one dimension.  Its scale function will be obtained in Theorem~\ref{THM58} and its speed measure will be explicitly defined in \S\ref{SEC53}.  A review of quasidiffusion will be outlined in \S\ref{SEC6}.  Note that a quasidiffusion is a semimartingale.  In particular we will distinguish Markov local times and  semimartingale local times for a quasidiffusion in Lemma~\ref{LM65},  so that the It\^o-Tanaka-Meyer formula is set up.  The correspondence between  SFH and quasidiffusion will be established in \S\ref{SEC7}.  %The endpoints of the state space will be classified in Feller's sense for these equivalent Markov processes in \S\ref{SEC73}.  
Finally the analytic characterization of SFHs in terms of Dirichlet forms will be completed in \S\ref{SEC2}.  

%Two additional conditions marked as (DK) and (DM) are assumed.  

\subsection*{Notations}
Let $\overline{\mathbb{R}}=[-\infty, \infty]$ be the extended real number system.  A set $E\subset \overline{\bR}$ is called a \emph{nearly closed subset} of $\overline{\bR}$ if $\overline E:= E\cup \{l,r\}$ is a closed subset of $\overline{\bR}$ where $l=\inf\{x: x\in E\}$ and $r=\sup\{x: x\in E\}$.  The point $l$ or $r$ is called the left or right endpoint of $E$.  Denote by $\overline{\mathscr K}$ the family of all nearly closed subsets of $\overline{\bR}$.  Set
\[
	\mathscr K:=\{E\in \overline{\mathscr K}: E\subset \bR\},
\]
and every $E\in \mathscr K$ is called a \emph{nearly closed subset} of $\bR$.  

Let $E$ be a locally compact separable metric space.  We denote by $C(E)$ the space of all real valued continuous functions on $E$.  In addition,  $C_c(E)$ is the subspace of $C(E)$ consisting of all continuous functions on $E$ with compact support and
\[
	C_\infty(E):=\{f\in C(E): \forall \varepsilon>0, \exists K\text{ compact},  |f(x)|<\varepsilon, \forall x\in E\setminus K\}.  
\]
The functions in $C_\infty(E)$ are said to be vanishing at infinity.  
Given an interval $J$,  $C_c^\infty(J)$ is the family of all smooth functions with compact support on $J$.  

The abbreviations CAF and PCAF stand for \emph{continuous additive functional} and \emph{positive continuous additive functional} respectively.

\section{Skip-free Hunt processes}\label{SEC5}

In this section we will introduce and study skip-free Hunt processes in one dimension.  They are defined by virtue of certain pathwise properties.    %Unless otherwise specified we would not adopt the notations in the previous sections.

\subsection{Skip-free Hunt processes}\label{SEC51}

We first put forward the state space of SFH.  
Let $E\in \overline{\mathscr K}$ be a nearly closed subset of $\overline{\bR}$ ended by $l$ and $r$. Write $[l,r]\setminus \overline{E}$ as a disjoint union of open intervals:
\begin{equation}\label{eq:51}
	[l,r]\setminus \overline{E}=\cup_{k\geq 1}(a_k,b_k).  
\end{equation}
We add a ceremony $\partial$ to $E$ and define $E_\partial:=E\cup \{\partial\}$.  More precisely,  $\partial$ is an additional isolated point when $E=\overline{E}$.  When $E=\overline E\setminus \{l\}$ or $E=\overline{E}\setminus \{r\}$,  $\partial$ is identified with $l$ or $r$.  When $E=\overline{E}\setminus \{l,r\}$,  $E_\partial$ is the one-point compactification of $E$. 

Now let 
\[
X=\left\{\Omega,  \mathscr F_t,  \theta_t,  X_t,  \mathbf{P}_x,  \zeta\right\}
\]
be a \emph{Hunt process} on $E_\partial$,  where $\{\mathscr F_t\}_{t\in [0,\infty]}$ is the minimum augmented admissible filtration and $\zeta=\inf\{t>0: X_t=\partial\}$ is the lifetime of $X$.  %Denote  by
%\[
%	\sF^0_\infty:=\sigma\{X_s: s<\infty\},\quad \sF^0_t:=\sigma\{X_s:s\leq t\}\text{ for }t<\infty
%\] 
%the minimum admissible filtration.  
The other notations and terminologies are standard and we refer readers to,  e.g.,  \cite[Appendix~A]{CF12}. 

\begin{definition}\label{DEF51}
Let $X$ be a Hunt process on $E_\partial$.  Then $X$ is called a \emph{skip-free Hunt process} if the following are satisfied:
\begin{itemize}
\item[(SF)] \emph{Skip-free property}: $(X_{t-} \wedge X_t,  X_{t-}\vee X_t)\cap E=\emptyset$ for any $t<\zeta$,  $\mathbf{P}_x$-a.s.  and all $x\in E$.   
\item[(SR)] \emph{Regular property}: $\mathbf{P}_x(T_y<\infty)>0$ for any $x, y\in E$,  where $T_y:=\inf\{t>0: X_t=y\}$ ($\inf \emptyset:=\infty$).   
\item[(SK)] There is \emph{no killing inside} in the sense that if $\mathbf{P}_x(\zeta<\infty)>0$ for $x\in E$,  then $l$ or $r$ does not belong to $E$ and $\mathbf{P}_x(X_{\zeta-}\notin E,  \zeta<\infty)=\mathbf{P}_x(\zeta<\infty)$.  
\end{itemize}
\end{definition}
\begin{remark}
\begin{itemize}\label{RM53}
\item[(1)] When $\overline{E}$ is a closed interval,  (SF) is identified with the continuity of all sample paths and $X$ is a regular diffusion with no killing inside;  see,  e.g.,  \cite[Chapter VII,  \S3]{RY99}. 
\item[(2)] (SF) particularly implies the following: Given states $x,y,z\in E$ with $x<y<z$ (or $x>y>z$),  and times $r<s$,  if $X_r=x$ and $X_s=z$,  then there is $t\in (r,s)$ such that $X_t=y$ or $X_{t-}=y$.   
To see this,  set $t:=\inf\{u>r: X_u>y\}$.  Then $r<t\leq s$,  $X_t\geq y$ and $X_{t-}\leq y$.  If $X_t>y$ and $X_{t-}<y$,  then $y\in (X_{t-}, X_t)$ as violates (SF).   
\item[(3)] (SR) is slightly different from the regular property for a diffusion in \cite[Chapter VII,  \S3]{RY99} (see also \cite[Chapter V,  (45.2)]{RW87}).    Here we impose $\mathbf{P}_x(T_y<\infty)>0$ for all $x\in E$,  while in \cite{RY99} this condition is only assumed for $x\in E\setminus \{l,r\}$.  This stronger assumption bars the possibility that $l$ or $r$ that contained in $E$ becomes an absorbing point.  (So we have to use strict irreducibility rather than irreducibility for the analytic characterization of SFHs in Theorem~\ref{THM53}.) See Corollary~\ref{COR55} for further discussion. 
\item[(4)] The process $X$ is called \emph{conservative} if $\mathbf{P}_x(\zeta=\infty)=1$ for all $x\in E$.  Note that if $l,r\in E$,  then $X$ is conservative due to (SK).  In general this definition is equivalent to that $\mathbf{P}_x(\zeta=\infty)=1$ holds for one $x\in E$.  To see this,  consider the case $l\notin E,  r\in E$ and the other cases can be argued analogously.  Lemma~\ref{LM54}~(4) tells us that $\zeta=T_l=\lim_{n\rightarrow \infty} T_{a_n}$ with $a_n\downarrow l$.  Take $x,y\in E$ with $a_n<x<y$.  It follows from Lemma~\ref{LM54}~(1) and the strong Markov property of $X$ that
\[
	\mathbf{P}_y(T_{a_n}<\infty)=\mathbf{P}_y(T_x<\infty)\mathbf{P}_x(T_{a_n}<\infty).  
\]
Letting $n\uparrow \infty$ we get that $\mathbf{P}_y(\zeta<\infty)=\mathbf{P}_y(T_x<\infty)\mathbf{P}_x(\zeta<\infty)$.  On account of (SR),  $\mathbf{P}_y(\zeta<\infty)>0$ is equivalent to $\mathbf{P}_x(\zeta<\infty)>0$.  Particularly,  $X$ is conservative if $\mathbf{P}_x(\zeta=\infty)=1$ holds for one $x\in E$.  
\item[(5)] When $l\notin E$ or $r\notin E$,  $X_{\zeta-}=\partial$ if $\zeta<\infty$.   In abuse of notations we will also write $X_{\zeta-}=l$ or $r$ to stand for that the convergence $\lim_{t\uparrow \zeta} X_t=\partial$ is along the decreasing or increasing direction.  Loosely speaking,  $l$ or $r$ or both of them are viewed as the ceremony.   
\end{itemize}

%It is also worth pointing out that this skip-free property is equivalent to that in \cite{BEPP08},  i.e. the condition (I) appearing in \S\ref{SEC46}.  To see this,  suppose $x_0:=X_{t-}<X_t=:z$ and $y\in (x_0, z)\cap E$.  Since $\lim_{u\uparrow t}X_u=X_{t-}=x_0<y$,  there exists $\varepsilon>0$ such that $X_u< y$ for any $t-\varepsilon<u<t$.  Set $r:=t-\varepsilon$ and $x:=X_r$.  Then $X_r=x,  X_t=z,  y\in (x,z)$ while for any $s\in (r,t)$,  $X_s\neq y$.  To the contrary,  let $r<t$ and $x:=X_r<X_t=:z$.  Take $y\in (x,z)\cap E$.  The skip-free property in Definition~\ref{DEF51} implies that there exists $s\in (r,t)$ such that $X_s=y$ or $X_{s-}=y$.  Only the latter case should be treated.  Note that $T_y\leq \hat{T}_y:=\inf\{t>0:X_{t-}=y\}$;  see, e.g., \cite[Theorem~A.2.3]{FOT11}.   
\end{remark}

In this section we always fix an SFH $X$ on $E_\partial$.    When there is no risk of ambiguity we would omit the subscript of $E_\partial$ and write $E$ for the state space of $X$.  For $a\in E$, define
\[
	T_{>a}:=\inf\{t>0: X_t>a\},\quad T_{<a}:=\inf\{t>0:X_t<a\}
\]
and 
\[
	T_{\geq a}:=\inf\{t>0: X_t\geq a\},\quad T_{\leq a}:=\inf\{t>0:X_t\leq a\}. 
\]
Set $T_{\neq a}:=\inf\{t>0: X_t\neq a\}=T_{>a}\wedge T_{<a}$.  
The lemma below states some useful facts concerning $X$. 

\begin{lemma}\label{LM54}
\begin{itemize}
\item[\rm (1)] For $x\in E$,
\[
	\mathbf{P}_x\left(\bigcup_{y,z\in E \text{ s.t. }x<y<z\text{ or }z<y<x} \{T_z< T_y\} \right)=0.  
\]
\item[\rm (2)] For $x,y \in E$ with $x>y$ (resp. $x<y$),  
\[
\mathbf{P}_x(T_{\leq y}=T_y)=1\quad  (\text{resp.  } \mathbf{P}_x(T_{\geq y}=T_y)=1). 
\]   
\item[\rm (3)] Let $x<y_n\uparrow y$ or $x>y_n\downarrow y $ with $x,y_n,y\in E$.  Then $$\mathbf{P}_x(\lim_{n\rightarrow \infty}T_{y_n}=T_y)=1. $$  
\item[\rm (4)] When $j\notin E$ for $j=l$ or $r$,  it holds $\mathbf{P}_x$-a.s.  for any $x\in E$ that 
\[
	T_j:=\lim_{n\rightarrow \infty}T_{a_n}=\left\lbrace
	\begin{aligned}
		&\zeta,\quad \text{if }\zeta<\infty,  X_{\zeta-}=j,\\
		&\infty,\quad \text{otherwise},
	\end{aligned}
	\right.
\]
where $a_n\rightarrow j$ is taken to be a monotone sequence contained in $E$.  Particularly,  
\[
	\zeta=T_l,\quad  (\text{resp. }  \zeta=T_r;\;  \zeta=T_l\wedge T_r)
\]
if $l\notin E, r\in E$ (resp.  $l\in E, r\notin E$; $l,r\notin E$).   
\item[\rm (5)] It holds that
\begin{equation}\label{eq:52}
	\mathbf{P}_x(T_{>x}=0)=\left\lbrace \begin{aligned}
		1,\quad x\neq a_k, \\
		0,\quad x=a_k,
	\end{aligned} \right.  \quad x\in E\setminus \{r\},
\end{equation}
and 
\begin{equation}\label{eq:53}
	\mathbf{P}_x(T_{<x}=0)=\left\lbrace \begin{aligned}
		1,\quad x\neq b_k, \\
		0,\quad x=b_k,
	\end{aligned} \right.  \quad x\in E\setminus \{l\},
\end{equation}
where $a_k,b_k$ are given in \eqref{eq:51}.  %Particularly,  for $x\in E\setminus \{l,r,a_k,b_k: k\geq 1\}$,  $\mathbf{P}_x(T_x=0)=1$.  
\item[(6)]  For any $x\in E$,  $\mathbf{P}_x(T_x=0)=1$.  
\end{itemize}
\end{lemma}
\begin{proof}
\begin{itemize}
\item[(1)] Let $\omega\in \Omega$ such that $T_z(\omega)<T_y(\omega)$ for $x,y,z\in E$ with  $x<y<z$.  We show that $(X_{t-}(\omega) \wedge X_t(\omega),  X_{t-}(\omega)\vee X_t(\omega))\cap E\neq  \emptyset$ for some $t$,  so that (SF) leads to the assertion.  To do this,  note that
\[
	T_{>y}\leq T_z<T_y\leq \hat{T}_y:=\inf\{t>0:X_{t-}=y\},
\]
where $T_y\leq \hat{T}_y$ is due to,  e.g.,  \cite[Theorem~A.2.3]{FOT11}.   Set $t:=T_{>y}(\omega)$.  By means of the c\`adl\`ag property of all paths,  we have 
\[
X_t(\omega)\geq y,  \quad X_{t-}(\omega)\leq y,
\]
while $t<T_y(\omega)\leq \hat{T}_y(\omega)$.  Hence $X_t(\omega)>y$ and $X_{t-}(\omega)<y$.  Therefore $y\in (X_{t-}(\omega), X_t(\omega))\cap E$.  
\item[(2)] We only treat the case $x>y$.  Note that $T_{\leq y}\leq T_y$.  Take $\omega\in \Omega$ such that $T_{\leq y}(\omega)<T_y(\omega)$.   It suffices to show that $y \in (X_{t-}(\omega) \wedge X_t(\omega),  X_{t-}(\omega)\vee X_t(\omega))$ for some $t$.  In fact,  set 
\[
	t:=T_{\leq y}(\omega)<T_y(\omega)\leq \hat{T}_y(\omega).  
\]
Mimicking the argument in the first assertion,  we get that $X_t(\omega)< y$ and $X_{t-}(\omega)>y$.  Hence $y\in (X_t(\omega), X_{t-}(\omega))$.  
\item[(3)] We only consider the case $x<y_n\uparrow y$.  In view of the first assertion,  $T_{y_n}$ is increasing and $T_{y_n}\leq T_y$,  $\mathbf{P}_x$-a.s.  Hence  $T:=\lim_{n\rightarrow \infty}T_{y_n}\leq T_y$.  The quasi-left-continuity of $X$ implies that
\[
	X_T=X_{\lim_{n\rightarrow \infty} T_{y_n}}=\lim_{n\rightarrow\infty}  X_{T_{y_n}}=y,\quad \mathbf{P}_x\text{-a.s. on }\{T<\infty\}. 
\]
Thus $T\geq T_y$ and we eventually obtain that $T=T_y$,  $\mathbf{P}_x$-a.s.   
\item[(4)] We only treat the case $j=l\notin E$.  Fix $x\in E$ and assume without loss of generality that $a_1<x$.  It follows from the quasi-left-continuity of $X$ that
\[
	X_{T_l}=\lim_{n\rightarrow \infty} X_{T_{a_n}}=l,\quad \mathbf{P}_x\text{-a.s. on }\{T_l<\infty\}.  
\]
Hence $T_l\geq \zeta$,  $\mathbf{P}_x$-a.s.  As a result,  $T_l<\infty$ implies $T_l=\zeta$ and $X_{\zeta-}=X_{T_l-}=l$, because $T_{a_n}\leq T_l<\infty$ leads to $T_{a_n}<\zeta$.  To the contrary suppose $\zeta<\infty$ and $X_{\zeta-}=l$.  It suffices to show $T_{a_n}<\infty$, so that 
$T_{a_n}\leq \zeta$ and thus $T_l\leq \zeta$,  as leads to $T_l=\zeta$ by means of $T_l\geq \zeta$.  To accomplish this, note that $X_0=x>a_n$ and since $\lim_{s\uparrow \zeta}X_s=l$,  there exists $s<\zeta$ such that $X_s<a_n$.  Remark~\ref{RM53}~(2) indicates that there exists $t\in (0,s)$ such that $X_t=a_n$ or $X_{t-}=a_n$.  For the former case it holds that $T_{a_n}\leq t$.  For the latter case,  $T_{a_n}\leq \hat{T}_{a_n}\leq t$.  Hence $T_{a_n}<\infty$ is concluded.  
\item[(5)] We first prove \eqref{eq:52} and then \eqref{eq:53} can be argued similarly.  When $x=a_k$,   the second assertion implies
\[
	T_{>x}=\inf\{t>0: X_t\geq b_k\}=T_{b_k}.  
\]
The right continuity of all paths leads to $T_{b_k}>0$,  $\mathbf{P}_{x}$-a.s.  Hence $\mathbf{P}_{x}(T_{>x}=0)=0$.  Next consider $x\neq a_k$.  We show that $X_{T_{>x}}=x$,  $\mathbf{P}_x$-a.s. on $\{T_{>x}<\infty\}$.  In fact, it follows from the right continuity of all paths that $X_{T_{>x}}\geq x$.  To the contrary,  take $E\ni x_n \downarrow x$.  Suppose $\omega\in \Omega$ such that $T_{>x}(\omega)<\infty$ and $X_{T_{>x}}(\omega)\geq x_n$.  
Set $t:=T_{>x}(\omega)$.  Then $X_{t-}(\omega)\leq x$ and $X_t(\omega)\geq x_n$.  Hence $x_{n+1}\in (X_{t-}(\omega), X_t(\omega))$.  (SF) implies that $X_{T_{>x}}< x_n$,  $\mathbf{P}_x$-a.s.  on $\{T_{>x}<\infty\}$.  Letting $n\uparrow \infty$,  we obtain that $X_{T_{>x}}\leq x$ and eventually $X_{T_{>x}}= x$ is concluded.  With this fact at hand,  we argue $\mathbf{P}_x(T_{>x}=0)=1$ by contradiction.  On account of Blumenthal 0-1 law,  $\mathbf{P}_x(T_{>x}=0)$ equals $0$ or $1$.  Suppose that this value is $0$.  Set
\[
	\Lambda=\{\omega\in \Omega: \exists \varepsilon>0\text{ s.t. }X_t(\omega)\leq x \text{ for all }t\leq \varepsilon\}=\{T_{>x}>0\}.  
\]
By the right continuity of all paths and the definition of $T_{>x}$,  it holds that
\[
\begin{aligned}
	0&=\mathbf{P}_x\left(T_{>x}<\infty,  \exists \varepsilon>0\text{ s.t. }X_{T_{>x}+t}\leq x\text{ for }0\leq t\leq \varepsilon \right) \\
	&=\mathbf{E}_x\left(1_\Lambda\circ \theta_{T_{>x}}; T_{>x}<\infty\right).  
\end{aligned}\]
Using strong Markov property, $X_{T_{>x}}=x$ on $\{T_{>x}<\infty\}$ and $\mathbf{P}_x(\Lambda)=1$,  we get
\begin{equation}\label{eq:54}
	0=\mathbf{E}_x\left(\mathbf{P}_{X_{T_{>x}}}(\Lambda); T_{>x}<\infty\right)=\mathbf{P}_x(T_{>x}<\infty).  
\end{equation}
However $T_{>x}<T_{x_n}$ and (SR) implies $\mathbf{P}_x(T_{x_n}<\infty)>0$.  Thus $\mathbf{P}_x(T_{>x}<\infty)>0$,  which contradicts with \eqref{eq:54}.  Therefore \eqref{eq:52} is concluded.  %Repeating this argument we can also obtain that $\mathbf{P}_x(T_{\geq x}=0)=\mathbf{P}_x(T_{\leq x}=0)=1$.
%Next we show $\mathbf{P}_x(T_{\geq x}=0)=1$ and another assertion $\mathbf{P}_x(T_{\leq x}=0)=1$ can be concluded analogously.  Since $T_{\geq x}\leq T_{>x}$,  it follows that $\mathbf{P}_x(T_{\geq x}=0)=1$ for $x\neq a_k$.  Now consider $x=a_k$.  We claim that $X_{T_{\geq x}}=x$ on $\{T_{\geq x}<\infty\}$.  Clearly  $X_{T_{\geq x}}\geq x$.  Suppose $X_{T_{\geq x}}>x$.  Then $T_{\geq x}=T_{>x}=T_{b_k}$ and thus $X_{T_{\geq x}}=b_k$.  Since $X_t<x=a_k$ for any $0<t<T_{\geq x}$,  it follows from the skip-free property that $X_{(T_{\geq x})-}=a_k$.  Hence 
%\[
%	T_x\geq T_{\geq x} \geq \hat{T}_x\geq T_x.
%\]
%As a result,  $T_x=T_{\geq x}$ and $X_{T_{\geq x}}=X_{T_x}=x$,  which leading to a contradiction.  Mimicking the argument proving $\mathbf{P}_x(T_{>x}=0)=1$ in the previous step, we can eventually obtain that  $\mathbf{P}_x(T_{\geq x}=0)=1$.  
%Finally we prove $\mathbf{P}_x(T_x=0)=1$.  Since $\mathbf{P}_x(T_{\geq x}=0)=\mathbf{P}_x(T_{\leq x}=0)=1$,  it follows that 
\item[(6)] We first consider $x\in E\setminus \{l,r, a_k,b_k: k\geq 1\}$. Since  $\mathbf{P}_x( T_{<x}=0)=\mathbf{P}_x(T_{>x}=0)=1$,  it follows that for $\mathbf{P}_x$-a.s.  $\omega\in \Omega$,  there exist two sequences $t_n\downarrow 0$ and $s_n\downarrow 0$ such that
\[
	X_{t_n}(\omega)< x,\quad X_{s_n}(\omega)> x. 
\]
It suffices to show that for each $n\geq 1$,  there exists $t< t_n$ such that $X_t(\omega)=x$.  Assume without loss of generality that $s_n<t_n$.  
In view of Remark~\ref{RM53}~(2),  there exists $t'\in (s_n,t_n)$ such that $X_{t'}(\omega)=x$ or $X_{t'-}(\omega)=x$.  For the former case take $t:=t'$.  For the latter case note that $T_x(\omega)\leq \hat{T}_x(\omega) \leq t'$.   Hence there exists $t\leq t'<t_n$ such that  $X_t(\omega)=x$. 

Next we prove the assertion for $x=a_k\in E$ and the other cases can be treated analogically.  If $x$ is isolated in $E$,  then $\mathbf{P}_x(T_x=0)=1$ is the consequence of Corollary~\ref{COR55}.  Assume that $x_0<x_1<\cdots x_n\uparrow x$ with $x_n\in E$ for $n\geq 0$.  Argue by contradiction and suppose $\mathbf{P}_x(T_x>0)=1$.  Note that (SR) implies $0<c:=\mathbf{P}_{x_0}(T_x<\infty)\leq 1$.  In view of 
\[
	1=\mathbf{P}_x(T_x>0)=\mathbf{P}_x(\lim_{\varepsilon\downarrow 0} \{T_x>\varepsilon\})=\uparrow \lim_{\varepsilon\downarrow 0} \mathbf{P}_x(T_x>\varepsilon),
\]
one can obtain $\varepsilon_0>0$ such that 
\begin{equation}\label{eq:45}
\mathbf{P}_x(T_x>\varepsilon_0)>1-c/3.
\end{equation}  
Set
\[
	\Lambda:=\{T_x>0\}=\{\omega\in \Omega: \exists \varepsilon>0,\text{s.t. } X_t<x,  0<t<\varepsilon\},\quad \mathbf{P}_x\text{-a.s.}
\]
Since $T_{x_n}=T_{\leq x_n}$,  $\mathbf{P}_x$-a.s.  due to the second assertion,  it follows that $T_{x_n}$ is decreasing, $\mathbf{P}_x$-a.s., and $\lim_{n\rightarrow \infty} T_{x_n}(\omega)=0$ for any $\omega\in \Lambda$.  These yield
\[
	1=\mathbf{P}_x(\lim_{n\rightarrow\infty}T_{x_n}<\varepsilon_0/2)=\uparrow \lim_{n\rightarrow \infty}\mathbf{P}_x(T_{x_n}<\varepsilon_0/2).  
\]
Hence there exists $N\in \bN$ such that for all $n>N$,  
\begin{equation}\label{eq:46}
\mathbf{P}_x(T_{x_n}<\varepsilon_0/2)>1-c/3.
\end{equation}
On account of \eqref{eq:45} and \eqref{eq:46},  we have
\[
	\mathbf{P}_x(T_{x_n}<\varepsilon_0/2,  T_x>\varepsilon_0)\geq \mathbf{P}_x(T_{x_n}<\varepsilon_0/2)-\mathbf{P}_x(T_x\leq \varepsilon_0)>1-\frac{2c}{3}.  
\]
On the event $\{T_{x_n}<\varepsilon_0/2,  T_x>\varepsilon_0\}$,  $T_x>T_{x_n}$ and thus $T_x=T_{x_n}+T_x\circ \theta_{T_{x_n}}$.  Using the strong Markov property,  one gets
\[
\begin{aligned}
\mathbf{P}_x(T_{x_n}<\varepsilon_0/2,  T_x>\varepsilon_0)&\leq \mathbf{P}_x(T_{x_n}<\varepsilon_0/2,  T_x\circ \theta_{T_{x_n}}>\varepsilon_0/2) \\
&=\mathbf{P}_x(T_{x_n}<\varepsilon_0/2)\mathbf{P}_{x_n}(T_x>\varepsilon_0/2).  
\end{aligned}\]
Hence for all $n>N$,  
\begin{equation}\label{eq:47}
	\mathbf{P}_{x_n}(T_x>\varepsilon_0/2)\geq \mathbf{P}_x(T_{x_n}<\varepsilon_0/2,  T_x>\varepsilon_0)>1-\frac{2c}{3}.  
\end{equation}
On the other hand,  the third assertion indicates that
\begin{equation}\label{eq:48}
	c=\mathbf{P}_{x_0}(\uparrow \lim_{n\rightarrow \infty}T_{x_n}=T_x,  T_x<\infty)=\lim_{n\rightarrow \infty} \mathbf{P}_{x_0}(T_{x_n}>T_x-\varepsilon_0/2,  T_x<\infty).  
\end{equation}
Since $T_{x_n}\leq T_x$,  $\mathbf{P}_{x_0}$-a.s.,  it follows from the strong Markov property and \eqref{eq:47} that for all $n>N$,  
\[
\begin{aligned}
 \mathbf{P}_{x_0}(T_{x_n}>T_x-\varepsilon_0/2,  T_x<\infty)&=\mathbf{P}_{x_0}(T_{x_n}>T_x-\varepsilon_0/2,  T_{x_n}<\infty) \\
 &=\mathbf{P}_{x_0}(T_x\circ \theta_{T_{x_n}}<\varepsilon_0/2,  T_{x_n}<\infty)  \\
 &=\mathbf{P}_{x_0}(T_{x_n}<\infty) \mathbf{P}_{x_n}(T_x<\varepsilon_0/2)<\frac{2c}{3},
\end{aligned}\]
as leads to a contradiction with \eqref{eq:48}.  
\end{itemize}
That completes the proof.  
\end{proof}

Due to the strong Markov property of $X$,  there is a constant $\kappa(x)\in [0,\infty]$ depending on $x\in E$ such that 
\[
\mathbf{P}_x(T_{\neq x}>t)=e^{-\kappa(x) t};
\]
see,  e.g.,  \cite[Proposition~2.19]{RY99}.  	(SR) bars the possibility of $\kappa(x)=0$.  If $\kappa(x)=\infty$,  then the process leaves $x$ at once.   When $\kappa(x)\in (0,\infty)$,  $x$ is called a \emph{holding point} because the process stays at $x$ for an exponential holding time before jumping.  

\begin{corollary}\label{COR55}
The family of holding points consists of all isolated points in $E$.  
\end{corollary}
\begin{proof}
Let $x\in E$ be not isolated.  Then \eqref{eq:52} and \eqref{eq:53} imply that $\mathbf{P}_x(T_{>x}=0)=1$ or $\mathbf{P}_x(T_{<x}=0)=1$.  Since $T_{\neq x}=T_{>x}\wedge T_{<x}$,  it follows that $\mathbf{P}_x(T_{\neq x}=0)=1$.  Hence $x$ is not a holding point.  Now consider an isolated point $x\in E$.  It has three possibilities: $x=l=a_p$,  $x=r=b_q$ or $x=a_p=b_q$ for some $p,q\geq 1$.  When $x=l=a_p$,  we have
\[
	T_{\neq l}=T_{\geq b_p}=T_{b_p}.  
\]
Thus $\mathbf{P}_l(T_{\neq l}=0)=\mathbf{P}_l(T_{b_p}=0)=0$ and $l$ is a holding point.  Another case $x=r=b_q$ can be argued similarly.  When $x=a_p=b_q$,  it suffices to note that  $T_{\neq x}=T_{\leq a_q}\wedge T_{\geq b_p}=T_{a_q}\wedge T_{b_p}$,  $\mathbf{P}_x$-a.s.,  and hence 
\[
\mathbf{P}_x(T_{\neq x}=0)\leq \mathbf{P}_x(T_{a_q}=0)+\mathbf{P}_x(T_{b_p}=0)=0.  
\] That completes the proof.   
\end{proof}

\subsection{Scale function of SFH}

%From now on we always assume that $X$ is good.  
Let $a,x,b\in E$ with $a<x<b$.  Set $H:=(a,b)\cap E$ and $\tau_H:=\inf\{t>0: X_t\notin H\}$.  In view of Lemma~\ref{LM54}~(2),
\[
\tau_H=T_a\wedge T_b,\quad \mathbf{P}_x\text{-a.s.}
\]
Repeating the argument in \cite[Proposition~3.1]{RY99},  we can obtain that $\mathbf{E}_x\tau_H<\infty$ and thus $\tau_H<\infty$,  $\mathbf{P}_x$-a.s.  
Particularly,  $\mathbf{P}_x(T_a<T_b)+\mathbf{P}_x(T_b<T_a)=1$.  The following result extends \cite[Proposition~3.2]{RY99}. 

\begin{theorem}\label{THM58}
There exists a continuous and strictly increasing real valued function $\bs$ on $E$ such that for any $a,x,b\in E$ with $a<x<b$,  
\begin{equation}\label{eq:55}
\mathbf{P}_x(T_b<T_a)=\frac{\bs(x)-\bs(a)}{\bs(b)-\bs(a)}. 
\end{equation}
If $\tilde{\bs}$ is another function with the same properties,  then $\tilde{\bs}=\alpha \bs +\beta$ with $\alpha>0$ and $\beta\in \bR$.  
\end{theorem}
\begin{proof}
We first consider the case $l,r\in E$.  Using the strong Markov property and repeating the argument in the first paragraph of the proof of \cite[Proposition~3.2]{RY99},  we can conclude
\[
\mathbf{P}_x(T_r<T_l)=\mathbf{P}_x(T_a<T_b)\mathbf{P}_a(T_r<T_l)+\mathbf{P}_x(T_b<T_a)\mathbf{P}_b(T_r<T_l).  
\]
Setting $\bs(x)=\mathbf{P}_x(T_r<T_l)$ we get the formula \eqref{eq:55}.  To show $\bs$ is strictly increasing,  take $l\leq x<y\leq r$ and note that 
\begin{equation}\label{eq:56}
	\bs(x)=\mathbf{P}_x(T_r<T_l)=\mathbf{P}_x(T_y<T_l,  T_r\circ \theta_{T_y}<T_l\circ \theta_{T_y})=\mathbf{P}_x(T_y<T_l) \bs(y).  
\end{equation}
Hence $\bs(x)\leq \bs(y)$.  We assert that $\bs(y)>\bs(l)=0$.  Argue by contradiction and suppose $\bs(y)=0$,  i.e.  $\mathbf{P}_y(T_r<T_l)=0$.  Set $\sigma_0:=0$,  and for $n\geq 0$, 
\[
\tau_{n+1}:=\inf\{t>\sigma_n: X_t=y\},\quad \sigma_{n+1}:=\inf\{t>\tau_{n+1}: X_t=l  \text{ or }r\}.  
\]	 
Then $A_n:=\{X_{\sigma_{n}}=r,  \sigma_{n}<\infty\}=\{\tau_n<\infty,  T_r\circ \theta_{\tau_n}<T_l\circ \theta_{\tau_n}\}$ and 
\[
	0<\mathbf{P}_l(T_r<\infty)=\mathbf{P}_l(\cup_{n\geq 1}A_n)\leq \sum_{n\geq 1}\mathbf{P}_l(A_n),
\]
where the first inequality is due to (SR).  However the strong Markov property yields
\[
	\mathbf{P}_l(A_n)=\mathbf{P}_l(\tau_n<\infty)\mathbf{P}_y(T_r<T_l)=0,
\]
as leads to a contradiction.  Hence $\bs(y)>0$ for any $y>l$.  Now take $l<x<y$ in \eqref{eq:56} and clearly,  $\bs(x)=\bs(y)$ amounts to $\mathbf{P}_x(T_y<T_l)=1$.  We can find that this is impossible by mimicking the argument proving $\bs(y)>\bs(l)$.  As a result,  $\bs$ is strictly increasing on $E$.  Let us turn to prove that $\bs$ is continuous on $E$.  
To accomplish this,  fix $a\in E$ and let $E\ni a_n\downarrow a$.  Take $x,b\in E$ and $a_1<x<b$.  In view of Lemma~\ref{LM54}~(3) and \eqref{eq:55},
\[
	\lim_{n\rightarrow \infty}\frac{\bs(x)-\bs(a_n)}{\bs(b)-\bs(a_n)}=\lim_{n\rightarrow \infty}\mathbf{P}_x(T_b<T_{a_n})=\mathbf{P}_x(T_b<T_a)=\frac{\bs(x)-\bs(a)}{\bs(b)-\bs(a)}.
\]
Hence $\lim_{n\rightarrow \infty}\bs(a_n)=\bs(a)$.  The left continuity can be argued analogously.  Therefore the continuity of $\bs$ on $E$ is eventually concluded.  

For the general case,  take $l_1,l_2,r_1,r_2\in E$ with $l_2<l_1<r_1<r_2$.  Put
\[
	\bs_1(x):=\mathbf{P}_x(T_{r_1}<T_{l_1}),\quad x\in (l_1,r_1)
\]
and 
\[
	\bs_2(x):=\mathbf{P}_x(T_{r_2}<T_{l_2}),\quad x\in (l_2,r_2).
\]
Mimicking the argument treating the case $l,r\in E$,  we can conclude that both $\bs_1$ and $\bs_2$ are continuous, strictly increasing and solving \eqref{eq:55} for $l_1\leq a<x<b\leq r_1$.  In addition,  it follows from the strong Markov property that for $x\in (l_1,r_1)$, 
\[
\begin{aligned}
\bs_2(x)&=\mathbf{P}_x(T_{l_1}<T_{r_1})\mathbf{P}_{l_1}(T_{r_2}<T_{l_2})+\mathbf{P}_x(T_{r_1}<T_{l_1})\mathbf{P}_{r_1}(T_{r_2}<T_{l_2}) \\
&=(c_{r_1}-c_{l_1})\bs_1(x)+c_{l_1},
\end{aligned}\]
where $c_{l_1}:=\mathbf{P}_{l_1}(T_{r_2}<T_{l_2})$ and $c_{r_1}:=\mathbf{P}_{r_1}(T_{r_2}<T_{l_2})$.  Then by a standard argument,  one can obtain a continuous and strictly increasing function $\bs$,  which is unique up to an affine transformation,  on $E$ such that \eqref{eq:55} is satisfied.  That completes the proof.  
\end{proof}
\begin{remark}\label{RM48}
Although defined up to an affine transformation,  the function $\bs$ in the preceding theorem is called the \emph{scale function} of $X$.  If $\bs$ may be taken equal to $x$ on $E$,  $X$ is said to be on its \emph{natural scale}.  Clearly,  $\tilde{X}=(\tilde{X}_t)_{t\geq 0}:=(\bs(X_t))_{t\geq 0}$ is an SFH on its natural scale on $\tilde{E}=\bs(E):=\{\bs(x):x\in E\}$.  It is worth pointing out that if $l\in E$ or $r\in E$,  then $\bs(l)\in \tilde{E}$ or $\bs(r)\in \tilde{E}$ is finite.  Particularly,  $\tilde{E}$ is a nearly closed subset of $\bR$ (not of $\overline{\bR}$),  i.e.  $\tilde{E}\in \mathscr K$. 
\end{remark}

We end this subsection with two corollaries concerning the scale function.  The first corollary shows that if $l\notin E$ is approachable in finite time,  $\bs(l):=\lim_{x\downarrow l}\bs(x)$ is finite.  The analogical assertion for $r$ holds as well.   

\begin{corollary}\label{COR59}
If $l\notin E$ and $\mathbf{P}_x(X_{\zeta-}=l, \zeta<\infty)>0$ for some $x\in E$,  then $\bs(l)>-\infty$.  
\end{corollary}
\begin{proof}
Argue by contradiction and suppose $\bs(l)=-\infty$.  Letting $a\downarrow l$ in \eqref{eq:55},  we have $\mathbf{P}_x(T_b<T_l)=1$.  Set
\[
	S_0:=0,\quad T_{n+1}:=\inf\{t>S_n:X_t=b\},\quad S_{n+1}:=\{t>T_{n+1}: X_t=x\},\quad n\geq 0.  
\]
Note that 
\[
\{X_{\zeta-}=l, \zeta<\infty\}\subset \cup_{n\geq 0} \{S_n<\infty,  T_l<T_{n+1}\}.  
\]
Since $T_{n+1}=S_n+T_b\circ \theta_{S_n}$ and $T_l=S_n+T_l\circ \theta_{S_n}$ on $\{S_n<\infty\}$,  it follows from the strong Markov property that
\[
	\mathbf{P}_x(S_n<\infty,  T_l<T_{n+1})=\mathbf{P}_x(S_n<\infty)\mathbf{P}_x(T_l<T_b)=0.
\]
This yields $\mathbf{P}_x(X_{\zeta-}=l,\zeta<\infty)=0$,  as violates the condition.  That completes the proof. 
\end{proof}

Another corollary extends \cite[Proposition~3.5]{RY99}.  Recall that $T_j=\inf\{t>0:X_t=j\}$ for $j=l$ or $r$ in $E$.  When $j\notin E$,  we have made the convention $T_j=\lim_{a\rightarrow j}T_a$.  
 Let $T:=T_l\wedge T_r$ and define a stopped process
\[
	X^T_t:=\left\lbrace
	\begin{aligned}
	& X_t,\quad t<T,   \\
	 & X_T, \quad t\geq T,
	 \end{aligned}\right. 
\]
where $X_T:=l$ (resp.  $r$) if $T=T_l<\infty$ (resp.  $T=T_r<\infty$).  The result below states that $\bs(X^T)$ is a (not necessarily continuous) local martingale.  Particularly,  the stopped process for $X$ on its natural scale is always a local martingale. 

\begin{corollary}
For any $x\in E$,  $\tilde{X}^T=(\bs(X^T_t))_{t\geq 0}$ is an $\{\sF_t,\mathbf{P}_x\}$-local martingale.  
\end{corollary}
\begin{proof}
Consider the case $l,r\in E$.  We may take $\bs(x)=\mathbf{P}_x(T_r<T_l)$.  We assert that $\tilde{X}^T_t=\mathbf{P}_x\left(T_r<T_l |\sF_t\right)$,  so that $\tilde{X}^T$ is a Doob martingale.  This trivially holds for $x=l$ or $r$.  It suffices to prove it for $x\in E\setminus \{l,r\}$.  In fact, the strong Markov property yields
\begin{equation}\label{eq:57}
	\tilde{X}^T_t=\mathbf{P}_{X_{t\wedge T}}(T_r<T_l)=\mathbf{P}_x\left( T_r\circ \theta_{t\wedge T}<T_l\circ \theta_{t\wedge T}| \sF_{t\wedge T}\right).
\end{equation}
Since $t\wedge T<T_l$ and $t\wedge T<T_r$,  it follows that $T_l=t\wedge T + T_l\circ  \theta_{t\wedge T}$ and $T_r=t\wedge T + T_r\circ  \theta_{t\wedge T}$.  These,  together with $\{T_r<T_l\}\in \sF_T$,  imply
\begin{equation}\label{eq:58}
	\tilde{X}^T_t=\mathbf{P}_x\left(T_r<T_l|\sF_{t\wedge T} \right)=\mathbf{P}_x\left(T_r<T_l|\sF_{t} \right).
\end{equation}

Next,  we consider the case $l\notin E$ and $r\in E$.  In view of Lemma~\ref{LM54}~(4),  $\zeta=T_l=\lim_{E\ni l_n\downarrow l}T_{l_n}$.  When $X$ is conservative (see Remark~\ref{RM53}~(4)),  $\{T_{l_n}\}$ is a sequence of $\{\sF_t\}$-stopping times that increases to $\infty$.  Mimicking the argument treating the case $l,r\in E$,  we can obtain that $\bs(X_{t\wedge T_{l_n}\wedge T_r})$ is a Doob martingale.  Hence $\tilde{X}^T$ is a local martingale.  When $X$ is not conservative,  in view of Corollary~\ref{COR59},  $\bs(l)>-\infty$.  Letting $b=r$ and $a\downarrow l$ in \eqref{eq:55},  we get
\[
\bs(y)=(\bs(r)-\bs(l)) \mathbf{P}_y(T_r<T_l) +\bs(l),  \quad y\in E. 
\]
When $t\wedge T\geq T_l$,  $\bs(X^T_t)=\bs(l)$.  It follows that
\[
	\bs(X^T_t)=(\bs(r)-\bs(l))\mathbf{P}_{X^T_t}(T_r<T_l) \cdot 1_{\{t\wedge T<T_l\}} +\bs(l).  
\]
The strong Markov property of $X$ yields
\[
\begin{aligned}
	\mathbf{P}_{X^T_t}(T_r<T_l) \cdot 1_{\{t\wedge T<T_l\}}&=\mathbf{P}_x\left(T_r\circ \theta_{t\wedge T}<T_l\circ \theta_{t\wedge T},  t\wedge T<T_l |\sF_{t\wedge T }\right) \\
	&=\mathbf{P}_x(T_r<T_l,  t\wedge T<T_l| \sF_{t\wedge T}) \\
	&=\mathbf{P}_x(T_r<T_l|\sF_{t}).
\end{aligned}\]
As a result,  $\tilde{X}^T$ is a Doob martingale and hence a local martingale. 

The other cases $l\in E, r\notin E$ and $l,r\notin E$ can be treated analogously.  That completes the proof.  
\end{proof}

\subsection{Speed measure of SFH}\label{SEC53}

In this subsection we define the speed measure $\mu$ for an SFH $X$ on its natural scale,  i.e.  $\bs(x)=x$. 

%We first consider a good SFH $X$ on its natural scale,   and turn to define its speed measure,  which will be denoted by $\mu$.  

\subsubsection{$l,r\notin E$}\label{SEC531}

Let us begin with the case $l,r\notin E$.   %and $\mathbf{P}_x(T_l=T_r=\infty)=1$ for all (equivalently,  one) $x\in E$.  Particularly,  $X$ is conservative.  
Take $a, b\in E$ with $a<b$ and define a function on $[a,b]$ as follows:
\[
	h_{a,b}(x):=\left\lbrace
	\begin{aligned}
	&\mathbf{E}_x(T_a\wedge T_b),\quad x\in [a,b]\cap E,  \\
	&\frac{h_{a,b}(b_k)-h_{a,b}(a_k)}{b_k-a_k}\cdot (x-a_k)+h_{a,b}(a_k),\quad x\in (a_k,b_k)\subset (a,b),  k\geq 1.  
	\end{aligned}
	\right.
\]
Recall that a real valued function $f$ defined on $(a,b)$ is called \emph{convex} if 
\[
f(tx+(1-t)y)\leq t f(x)+(1-t)f(y),\quad \forall 0\leq t\leq 1,  x,y \in (a,b).  
\]
It is called \emph{concave} if $-f$ is convex.  For a convex function $f$,  both its left derivative $f'_-$ and right derivative $f'_+$ are increasing,  respectively left and right continuous,  and the set $\{x: f'_-(x)\neq f'_+(x)\}$ is at most countable.  Particularly,  the second derivative $f''$ of $f$ in the sense of distribution is a positive Radon measure on $(a,b)$;  see, e.g.,  \cite[Appendix~\S3]{RY99}.   We have this.

\begin{lemma}\label{LM511}
$h_{a,b}$ is concave on $(a,b)$.  
\end{lemma}
\begin{proof}
For convenience write $h$ for $h_{a,b}$.  It suffices to show that for  $x,y,z\in [a,b]$ with $x<y<z$,  
\begin{equation}\label{eq:59}
h(y)\geq \frac{z-y}{z-x}h(x)+\frac{y-x}{z-x} h(z).  
\end{equation}
Firstly take $x,y,z\in E$ with $a<x<y<z<b$ and set $T_1:=T_a\wedge T_b$ and $T_2:=T_x\wedge T_z$.  It follows from Lemma~\ref{LM54}~(2) that $T_2<T_1$,  $\mathbf{P}_y$-a.s.  Using the strong Markov property and Theorem~\ref{THM58},   we get
\begin{equation}\label{eq:511}
\begin{aligned}
	h(y)&=\mathbf{E}_y(T_2+T_1\circ \theta_{T_2}) \\
	&=\mathbf{E}_yT_2+\frac{z-y}{z-x}h(x)+\frac{y-x}{z-x} h(z)\\
	&>\frac{z-y}{z-x}h(x)+\frac{y-x}{z-x} h(z).  
\end{aligned}\end{equation}
Hence \eqref{eq:59} holds for $x,y,z\in E$.   When $x=a$ and $z<b$,  noting $h(a)=0$ and $T_z<T_b$ we have
\[
	h(y)\geq \mathbf{E}_y(T_a\wedge T_b,  T_z<T_a)\geq \mathbf{P}_y(T_z<T_a)\mathbf{E}_z (T_a\wedge T_b).
\] 
Thus \eqref{eq:59} still holds true.  Analogously \eqref{eq:59} holds for $z=b$ and $x>a$.  The case $x=a, z=b$ is obvious.  

Now we consider the general case $x_0,y_0,z_0\in (a,b)$ with $x_0<y_0<z_0$.  Set $\dagger_-:=\sup\{a\in E: a<\dagger_0\}$ and $\dagger_+:=\inf\{a\in E: a>\dagger_0\}$,  where $\dagger$ stands for $x,y$ or $z$,  and clearly $h(\dagger_0)=t_{\dagger_0} h(\dagger_-)+(1-t_{\dagger_0})h(\dagger_+)$ for $t_{\dagger_0}:=\frac{\dagger_+-\dagger_0}{\dagger_+-\dagger_-}$ for $\dagger_0\notin E$ and $t_{\dagger_0}:=0$ for $\dagger_0\in E$.  The first step has proved that \eqref{eq:59} holds for $x=x_\pm$,  $y=y_-$ and $z=z_-$.  Using $h(x_0)=t_{x_0} h(x_-)+(1-t_{x_0})h(x_+)$,  one  can easily obtain \eqref{eq:59} for $x=x_0, y=y_-$ and $z=z_-$.  Analogously \eqref{eq:59} holds for $x=x_0, y=y_-$ and $z=z_+$.  Using $h(z_0)=t_{z_0} h(z_-)+(1-t_{z_0})h(z_+)$,  we find that \eqref{eq:59} holds for $x=x_0, y=y_-$ and $z=z_0$.  Analogously it holds also for $x=x_0,  y=y_+$ and $z=z_0$. Finally by means of $h(y_0)=t_{y_0} h(y_-)+(1-t_{y_0})h(y_+)$ we eventually conclude \eqref{eq:59} for $x=x_0,y=y_0$ and $z=z_0$.  That completes the proof. 
\end{proof}

Take $a',b'\in E$ with $a'<a<b<b'$ and define $h_{a',b'}$ analogously.  It is easy to compute that
\[
	h_{a',b'}(x)=h_{a,b}(x)+\frac{b-x}{b-a}h_{a',b'}(a)+\frac{x-a}{b-a}h_{a',b'}(b),\quad x\in (a,b),
\]
so that the restriction of Radon measure $-h''_{a',b'}$ to $(a,b)$ is identified with $-h''_{a,b}$.  Define the speed measure $\mu$ on $(l,r)$  as follows:
\begin{equation}\label{eq:510}
	\mu|_{(a,b)}:=-\frac{1}{2}h''_{a,b},\quad \forall a,b\in E\text{ s.t. } l<a<b<r.  
\end{equation}
Noting the following result,  we denote the restriction of $\mu$ to $E$ still by $\mu$.  

\begin{lemma}\label{LM512}
Let $\mu$ be defined as \eqref{eq:510}.  Then $\mu$ is a positive Radon measure with $\text{supp}[\mu]=E$.  
\end{lemma}
\begin{proof}
Clearly,  $\mu$ is a positive Radon measure and $\mu((a_k,b_k))=0$ for any $k\geq 1$.  Take $(\alpha,\beta)\subset (l,r)$ such that $\mu((\alpha,\beta))=0$.  We need to show $$(\alpha,\beta)\cap E=\emptyset,$$ so that the support of $\mu$ is $E$.  To do this take $(a,b)\supset [\alpha, \beta]\supset (\alpha,\beta)$ with $a,b\in E$ and let $h:=h_{a,b}$.  Then $h''=0$ on $(\alpha,\beta)$ implies that there is a constant $k$ depending on $\alpha,  \beta$ such that 
\begin{equation}\label{eq:512}
	h(x)-h(\alpha)=k(x-\alpha),\quad x\in [\alpha,\beta]. 
\end{equation}
Argue by contradiction and suppose $x\in (\alpha, \beta)\cap E$.  Using \eqref{eq:511} and repeating the argument in the second paragraph of the proof of Lemma~\ref{LM511},  we can conclude
\[
	h(x)>\frac{\beta-x}{\beta-\alpha} h(\alpha)+\frac{x-\alpha}{\beta-\alpha} h(\beta).  
\]
This violates \eqref{eq:512}.  That completes the proof.  
\end{proof}

\subsubsection{$l\in E$ but $r\notin E$}\label{SEC532}

Take $b\in E$ with $l<b<r$.  Using $\mathbf{P}_l(T_b<\infty)>0$ and repeating the argument in \cite[Chapter VII,  Proposition~3.1]{RY99},   one has $\mathbf{E}_xT_b<\infty$ for all $x\in [l,b]$ and particularly,  $\mathbf{P}_x(T_b<\infty)=1$. Define
\[
	h_b(x):=\left\lbrace
	\begin{aligned}
	&\mathbf{E}_x(T_b),\quad x\in [l,b]\cap E,  \\
	&\frac{h_{b}(b_k)-h_{b}(a_k)}{b_k-a_k}\cdot (x-a_k)+h_{b}(a_k),\quad x\in (a_k,b_k)\subset (l,b),  k\geq 1.  
	\end{aligned}
	\right.
\]
Mimicking Lemma~\ref{LM511},  we can obtain that $h_b$ is concave on $(l, b)$.  Since $h_b$ is decreasing on $[l,b)$,  the extended function obtained by letting $h_b(x):=h_b(l)$ for $x\leq l$ is concave on $(-\infty,  b)$.  Take another $b'\in E$ such that $b<b'$ and let $h_{b'}$ be the concave function defined on $(-\infty, b')$ analogously.  Then
\[
	h_{b'}(x)=h_b(x)+\mathbf{E}_b(T_{b'}).  
\]
Hence the restriction of $-h''_{b'}$ to $(-\infty, b)$ is identified with $-h''_b$.  Define the speed measure $\mu$ on $(-\infty, r)$ as
\begin{equation}\label{eq:513}
	\mu|_{(-\infty, b)}:=-\frac{1}{2}h''_b,\quad  r>b\in E.  
\end{equation}
As an analogue of Lemma~\ref{LM512},  $\mu$ a positive Radon measure with $\text{supp}[\mu]=E$.  The restriction of $\mu$ to $E$ is still denoted by $\mu$.  

\subsubsection{$l\notin E$ and $r\in E$}\label{SEC533}

As an analogue of $h_b$ in \S\ref{SEC532},  we set a function $h_a$ on $(a,  \infty)$ for $a\in E$ with $l<a<r$:
\[
	h_a(x):=\left\lbrace
	\begin{aligned}
	&\mathbf{E}_x(T_a),\quad x\in [a,r]\cap E,  \\
	&\frac{h_{b}(b_k)-h_{b}(a_k)}{b_k-a_k}\cdot (x-a_k)+h_{b}(a_k),\quad x\in (a_k,b_k)\subset (a,r),  k\geq 1, \\
	&\mathbf{E}_rT_a,\quad x\geq r.  
	\end{aligned}
	\right.
\]
Note that $h_a$ is concave and $-h''_a$ extends consistently to $(l, \infty)$ as $a\downarrow l$.  Define the speed measure as this limit,  i.e. 
\begin{equation}\label{eq:514}
	\mu|_{(a,\infty)}:=-\frac{1}{2}h''_a,\quad l< a\in E. 
\end{equation}
The support of $\mu$ is $E$ and we denote still by $\mu$ the restriction of $\mu$ to $E$.  

\subsubsection{$l, r\in E$}

Let $h_{l,r}:=h_{a,b}$ with $a=l, b=r$ as in \S\ref{SEC531},  $h_r:=h_b$ with $b=r$ as in \S\ref{SEC532} and $h_l:=h_a$ with $a=l$ as in \S\ref{SEC533}.  Then $-h''_{l,r},  -h''_l$ and $-h''_r$ are all positive Radon measures on $(l,r)$,  $[l, r)$ and $(l,r]$ respectively.  

\begin{lemma}
The restrictions of $-h''_{l,r},  -h''_l$ and $-h''_r$ to $(l,r)$ are all identified.  
\end{lemma}
\begin{proof}
Take $x\in (l,r)$.  We have $h_l(x)=\mathbf{E}_x(T_l; T_l<T_r)+\mathbf{E}_x(T_l; T_l>T_r)$.  Note that
\[
\begin{aligned}
	\mathbf{E}_x(T_l; T_l>T_r)&=\mathbf{E}_x(T_r+T_l\circ \theta_{T_r}; T_l>T_r) \\
	&=\mathbf{E}_r(T_r; T_l>T_r)+\mathbf{E}_rT_l \cdot \frac{x-l}{r-l}.  
\end{aligned}\]
It follows that $h_l(x)=\mathbf{E}_x(T_l\wedge T_r)+\mathbf{E}_rT_l \cdot \frac{x-l}{r-l}=h_{l,r}(x)+\mathbf{E}_rT_l \cdot \frac{x-l}{r-l}$.  Hence $-h''_l=-h''_{l,r}$ on $(l,r)$.  Another identity $-h''_r=-h''_{l,r}$ on $(l,r)$ can be obtained similarly.  That completes the proof. 
\end{proof}

Due to this lemma we define the speed measure $\mu$ on $[l,r]$ as follows:
\begin{equation}\label{eq:515}
	\mu|_{[l,r)}:=-\frac{1}{2}h''_r,\quad \mu|_{(l,r]}:=-\frac{1}{2}h''_l. 
\end{equation}
Clearly $\mu$ is a positive Radon measure with $\text{supp}[\mu]=E$. 

\subsubsection{Summary}

We summarize the definition of speed measure for an SFH as follows.  

\begin{definition}\label{DEF514}
\begin{itemize}
\item[(1)] Let $X$ be an SFH on its natural scale, i.e. $\bs(x)=x$. The measure $\mu$ defined as \eqref{eq:510},  \eqref{eq:513},  \eqref{eq:514} and \eqref{eq:515} for the cases mentioned above respectively is called the \emph{speed measure} of $X$. 
\item[(2)] For an SFH $X$ with scale function $\bs$,  let $\tilde{\mu}$ be the speed measure of $\tilde{X}=(\bs(X_t))_{t\geq 0}$.  Then the image measure $\mu:=\tilde\mu\circ \bs^{-1}$ of $\tilde\mu$ under the map $\bs$ is called the \emph{speed measure} of $X$.  
\end{itemize}
\end{definition}

In general the speed measure of an SFH $X$ is a positive Radon measure on $E$ with full support.  As an extension of the classical characterization for regular diffusions,  it will turn out in Theorem~\ref{THM71} that an SFH on its natural scale can be always constructed from a Brownian motion,  by way of a time change using this speed measure.  Particularly,  an SFH is uniquely determined by its scale function and speed measure;  see Corollary~\ref{COR64}.  

\section{Quasidiffusions}\label{SEC6}

This section is devoted to introducing a special family of SFHs that have been widely studied.  

\subsection{Quasidiffusions as standard processes}\label{SEC61}

Let $m$ be an extended real valued, right continuous,  (not necessarily strictly) increasing and non-constant function on $\bR$,  and set $m(\infty):=\lim_{x\uparrow \infty}m(x),  m(-\infty):=\lim_{x\downarrow -\infty}m(x)$.   Put 
\begin{equation}\label{eq:51-2}
\begin{aligned}
	&l_0:=\inf\{x\in \bR: m(x)>-\infty\},\quad r_0:=\sup\{x\in \bR: m(x)<\infty\},  \\
	&l:=\inf\{x>l_0: m(x)>m(l_0)\},\quad r:=\sup\{x<r_0: m(x)<m(r_0-)\}.  
\end{aligned}\end{equation}
To avoid trivial case assume that $l<r$.  Define
\[
E_m:=\{x\in [l,r]\cap (l_0,r_0): \exists \varepsilon>0\text{ s.t. }m(x-\varepsilon)<m(x+\varepsilon)\}.  
\]
The function $m$ corresponds to a measure on $\bR$,  still denoted by $m$ if no confusions caused.  %Note that $m(\{l_0\})=\infty$ (resp.  $m(\{r_0\})=\infty$) if $l_0>-\infty$ and $m(l_0+)>-\infty$ (resp.  $r_0<\infty$ and $m(r_0-)<\infty$).  %Clearly the restriction of $m$ to $E_m$ is a.  %We denote this restricted measure still by $m$.

\begin{lemma}\label{LM51}
The set $E_m$ is a nearly closed subset of $\bR$ ended by $l$ and $r$, and the restriction of $m$ to $E_m$ is a  positive Radon measure with full support.
\end{lemma}
\begin{proof}
Let $\overline{E}_m:=E_m\cup \{l,r\}$.  For $x\in [l,r]\setminus \overline E_m$,  $m(x-\varepsilon)=m(x+\varepsilon)$ for some $\varepsilon>0$.  Hence $(x-\varepsilon,x+\varepsilon)\subset [l,r]\setminus \overline{E}_m$ whenever $x-\varepsilon>l$ and $x+\varepsilon<r$.  As a result,  $\overline{E}_m$ is closed in $\overline{\bR}$.   If $l\notin E_m$,  then we must have $l=l_0$ and $l$ can be approximated by points in $E_m$ from the right.  Analogical fact holds for $r$.  Since $E_m\subset \bR$, we can eventually conclude that $E_m\in \mathscr K$ and $E_m$ is ended by $l$ and $r$.  The second assertion for $m$ is obvious.  That completes the proof. 
\end{proof}
\begin{remark}
There may appear various cases for the positions of $r_0$ and $r$:
\begin{itemize}
\item[(1)] $r=r_0=\infty$ whenever $m(\infty)\leq  \infty$ and $m(x)<m(\infty)$ for any $x<\infty$.
\item[(2)] $r<r_0=\infty$ whenever $m(x)=m(\infty)<\infty$ for $x\in [r, \infty)$ and $m(x)<m(r)$ for $x<r$.  In this case $r\in E_m$.  
\item[(3)] $r=r_0<\infty$ whenever $m(x)=\infty$ for $x\in [r,\infty)$ and $m(x)<m(r-)$ for $x<r$.  In this case $r\notin E_m$.  
\item[(4)] $r<r_0<\infty$ whenever $m(x)=\infty$ for $x\in [r_0,\infty)$,  $m(x)=m(r)<\infty$ for $x\in [r, r_0)$ and $m(x)<m(r-)$ for $x<r$.  In this case $r\in E_m$.  
\end{itemize}
The positions of $l_0$ and $l$ can be argued similarly.  
\end{remark}

Let $W=(W_t, \mathscr{F}^W_t, \mathbf{P}_x)$ be a Brownian motion on $\bR$ and $\ell^W(t,x)$ be its local time normalized such that for any bounded Borel measurable function $f$ on $\bR$ and $t\geq 0$, 
\[
	\int_0^t f(W_s)ds=2\int_\bR \ell^W(t,x)f(x)dx.  
\]
Define $S_t:=\int_{\bR} \ell^W(t,x)m(dx)$ for $t\geq 0$ and
\[
	T_t:=\inf\{u>0: S_u>t\},\quad t\geq 0.  
\]
Then $T=(T_t)_{t\geq 0}$ is a strictly increasing (before $\zeta$ defined as below), right continuous family of $\sF_t^W$-stopping times with $T_0=0$,  $\mathbf{P}_x$-a.s.   Define
\[
	\sF_t:=\sF^W_{T_t},\quad \zeta:=\inf\{t>0: W_{T_t}\notin (l_0,r_0)\},\quad X_t:=W_{T_t}, \; 0\leq t<\zeta.  
\]
Then $\{X_t, \sF_t, (\mathbf{P}_x)_{x\in E_m},\zeta\}$,  called \emph{a quasidiffusion with speed measure $m$},  is a \emph{standard process} with state space $E_m$ and lifetime $\zeta$; see \cite{K86}.   The definition of standard process is referred to in,  e.g.,  \cite{BG68}.  

\begin{lemma}\label{LM515}
Let $\tau:=\inf\{t>0: W_t\notin (l_0,r_0)\}$.  Then $\zeta=S_{\tau-}$,  $\mathbf{P}_x$-a.s.,  $x\in E_m$.  Furthermore,  the following hold:
\begin{itemize}
\item[(1)] When $\tau=\infty$,  i.e.  $|l_0|=|r_0|=\infty$, 
\[
\begin{aligned}
	 &S_u<\infty \;  (u<\infty),  \quad \zeta=\lim_{u\uparrow \infty}S_u=\infty;  \\
	 &T_t<\infty\;  (t<\infty), \quad T_\infty:=\lim_{t\uparrow \infty}T_t=\infty.  
\end{aligned}\]
\item[(2)] When $\tau<\infty$,  i.e.  $|l_0|<\infty$ or $|r_0|<\infty$,
\[
\begin{aligned}
	& S_u<\infty\; (u<\tau),\quad S_u=\infty\; (u>\tau),\quad \zeta=S_{\tau-}\leq \infty;  \\
	& T_t<\tau\;  (t<\zeta),\quad T_t=\tau\;  (t\geq \zeta\text{ if }\zeta<\infty).  
\end{aligned}
\]
\end{itemize}
\end{lemma}
\begin{proof}
Consider first that $\tau=\infty$.  Clearly $m$ is a Radon measure on $\bR$.  Hence \cite[Chapter X,  Corollary~2.10]{RY99} yields $S_u<\infty$ for $u<\infty$.  Since $\ell^W(\infty,x)=\infty$ for any $x\in \bR$,  it follows that $\lim_{u\uparrow \infty}S_u=\infty$.  Particularly $T_t<\infty$ for $t<\infty$ and $T_\infty=\infty$.  Note that $\zeta=\inf\{t>0: T_t=\infty\}$ by its definition.  As a result,  $\zeta=\infty$.  

Now consider $\tau<\infty$.  Without loss of generality we only treat the case $l_0>-\infty,  r_0=\infty$ and $W_\tau=l_0$.  The other cases can be argued similarly.  
The fact $S_u<\infty$ for $u<\tau$ can be obtained by means of \cite[Chapter X, Corollary~2.10]{RY99} when the zero extension of $m|_{(l_0,\infty)}$ to $\bR$ is a Radon measure.  Otherwise set $\tau_n:=\inf\{t>0: W_t\notin (l_n,\infty)\}$,  where $l_n\downarrow l_0$,  and  \cite[Chapter X, Corollary~2.10]{RY99} implies $S_u<\infty$ for $u<\tau_n$.  Since $\tau_n\uparrow \tau$,  it follows that $S_u<\infty$ for $u<\tau$.  
To prove $S_{\tau+\varepsilon}=\infty$ for $\varepsilon>0$,  note that $l_0>-\infty$ must lead to $m(\{l_0\})=\infty$ or $m((l_0, l_0+\delta))=\infty$ for $\delta>0$.   Since $W_0=x,  W_\tau=l_0$,  it follows that $\ell^W(\tau+\varepsilon, y)>0$ for any $y\in [l_0,  x)$.  Using the continuity of $\ell^W(\tau+\varepsilon, \cdot)$,  we get
\[
	S_{\tau+\varepsilon}\geq \int_{[l_0,l_0+\delta)} \ell^W(\tau+\varepsilon, y)m(dy)=\infty.  
\]
Particularly $T_t<\tau$ for $t<S_{\tau-}$ and for $t\geq S_{\tau-}$ (if $S_{\tau-}<\infty$),  $T_t=\tau$.  Finally it suffices to prove $\zeta=S_{\tau-}$.  In fact,  $T_t<\tau$ for $t<S_{\tau-}$ implies $\zeta\geq S_{\tau-}$.  To the contrary we have $T_t=\tau$ if $t:=S_{\tau-}<\infty$.  Hence $W_{T_t}=l_0$, as leads to $\zeta\leq S_{\tau-}$.  Therefore $\zeta=S_{\tau-}$ is eventually concluded.  
That completes the proof. 
\end{proof}

A quasidiffusion satisfies (SF) and (SR) in Definition~\ref{DEF51}; see \cite{G75, K86,  BK87}.  
Attaching the ceremony $\partial$ to $E_m$ by the same way as in \S\ref{SEC51},  one can find in Theorem~\ref{THM71} that a quasidiffusion is,  in fact,  a Hunt process on $E_m$. 
However it admits killing at $l$ or $r$ whenever $-\infty<l_0<l$ or $r<r_0<\infty$; see Remark~\ref{RM73}.  To bar this possibility we will assume in the next section that 
\begin{itemize}
\item[(QK)] If $l_0>-\infty$ (resp.  $r_0<\infty$),  then $l=l_0$  (resp.  $r=r_0$).  
\end{itemize}
In other words,  under the assumption \text{(QK)},  a quasidiffusion is an SFH on $E_m$. 

\subsection{Markov local times}

Define 
\[
\begin{aligned}
&\ell(t,x):=\ell^W(T_t,x),\quad 0\leq t<\zeta, x\in E_m,  \\
&\ell(t,x):=\lim_{s\uparrow \zeta}\ell(s,x),\quad t\geq \zeta,  x\in E_m,
\end{aligned}
\]
which are called the \emph{local times} of $X$ in \cite{BK87}.  To avoid ambiguity with semimartingale local times,  we call them the \emph{Markov local times} of $X$.  Note that $\ell$ is jointly continuous in $t$ and $x$, and $\ell(\cdot, x)$ increases at $t$ if and only if $T_t<T_\infty$ and $X_t=x$ or $X_{t-}=x$.  Moreover,  by virtue of \cite[Chapter V,  Proposition~1.4]{RY99},  it holds a.s.  that for any bounded measurable function $f$ and $0\leq t<\zeta$,
\begin{equation}\label{eq:61}
\int_0^t f(X_s)ds=\int_{E_m}\ell(t,x)f(x)m(dx);
\end{equation}
see also \cite[(2.1)]{BK87}.  In view of Lemma~\ref{LM515},  $T_t\leq \tau$ for any $t\geq 0$.  We set $\ell(t,j):=\ell^W(T_t,j)\equiv 0$ for $j=l_0$ or $r_0$ whenever $l_0$ or $r_0$ is finite.  

\begin{remark}
When $\tau<\infty$ and $\zeta<\infty$,  we have $\ell(\zeta,x)=\ell^W(T_\zeta, x)$ for $x\in E_m$ even if $T_{\zeta-}<T_\zeta$.  In fact,  for $u\in (T_{\zeta-}, T_\zeta)$,  $S_u$ is constant and hence $W_u\notin E_m$.  This yields $\ell^W(T_{\zeta-},x)=\ell^W(T_\zeta, x)$ for $x\in E_m$.  Particularly,  $\ell(\zeta,x)=\ell^W(T_\zeta,x)$.  
\end{remark}

As a time change of Brownian motion,  $X$ is always a semimartingale by letting $X_t:=W_{T_\zeta}$ for $t\geq \zeta$ if $\zeta<\infty$; see,  e.g., \cite{M78}.  In what follows we present its semimartingale decomposition obtained in \cite{BK87}. 
Let $\tilde{E}_m:=E_m\cup \left(\{l_0,r_0\}\cap \bR\right)$,  which is closed in $\bR$.  Then $\bR\setminus \tilde{E}_m$ is the union of mutually disjoint open intervals $I_k:=(a_k,b_k)$,  where $k$ runs through a subset $K$ of $\bN$ such that $0\in K$ (resp.  $1\in K$) if and only if the left (resp.  right) endpoint point of some interval is $-\infty$ (resp.  $\infty$).   Set $I_0:=(-\infty,  b_0)$ (resp.  $I_1:=(a_1,\infty)$) whenever $0\in K$ (resp.  $1\in K$), and $K':=K\setminus \{0,1\}$.  Note that $0\in K$ amounts to either $l_0>-\infty$ or $l>l_0=-\infty$.  In the former case $b_0=l_0$ and in the latter case $b_0=l$.  This holds analogously for $1\in K$.  For convenience we make the convention $b_0:=-\infty$ or $a_1:=\infty$ if $0\notin K$ or $1\notin K$,  and $\ell(\cdot, \pm\infty):=0$.  
 Let $A_{k,i}$ (resp.  $B_{k,i}$) be the time of the $i$-th jump of $X$ from $a_k$ to $b_k$ (resp. from $b_k$ to $a_k$) if this jump occurs,  otherwise let it be equal to $\infty$.  For $k\in K'$ put
\[
	A_k(t):=\sum_{i=1}^\infty 1_{[A_{k,i},\infty)}(t),\quad B_k(t):=\sum_{i=1}^\infty 1_{[B_{k,i},\infty)}(t),\quad t\geq 0,
\]
i.e.  the number of jumps from $a_k$ to $b_k$ or from $b_k$ to $a_k$ before time $t$.  The following result is due to \cite{BK87}.  

\begin{theorem}\label{THM517}
\begin{itemize}
\item[\rm(1)] For $k\in K'$,  $A_k(t)-\lambda_k \ell(t,a_k)$ and $B_k(t)-\lambda_k \ell(t,b_k)$ are local $\sF_t$-martingales,  where $\lambda_k:=(b_k-a_k)^{-1}$.  
\item[\rm(2)] $X$ is an $\sF_t$-semimartingale and admits the representation $X=M^c+M^d+A$,  where 
\[
	M^c_t:=\int_0^{T_t} 1_{\tilde{E}_m}(W_s)dW_s
\]
is a continuous local martingale,  
\[
	M^d_t:=\sum_{k\in K'}\left([(b_k-a_k)A_k(t)-\ell(t,a_k)]-[(b_k-a_k)B_k(t)-\ell(t,b_k)] \right)
\]
is a purely discontinuous local martingale and $A_t=\ell(t,b_0)-\ell(t,a_1)$ is the adapted continuous process with locally integrable variation.  
\end{itemize}
\end{theorem}
\begin{remark}\label{RM64}
\begin{itemize}
\item[(1)]  We should emphasize that,  as a standard process,  $X_t$ is defined as the ceremony $\partial$ for $t\geq \zeta$.  However viewed as a semimartingale in this theorem,  $X_t$ is by definition equal to $W_{T_\zeta}$ for $t\geq \zeta$.  To see this difference,  consider the special case $-\infty<l_0<l<r<r_0<\infty$.  The lifetime $\zeta$ is the first time of $(W_{T_t})_{t\geq 0}$ leaving $(l_0,r_0)$ and $W_{T_t}=W_{T_\zeta}=l_0$ or $r_0$ for $t\geq \zeta$ corresponding to $W_{T_{\zeta-}}=l$ or $r$.  As a standard process,  $X_{\zeta-}=W_{T_{\zeta-}}=l$ or $r$ while $X$ jumps to the ceremony $\partial$ at lifetime $\zeta$.  To be more exact,  the semimartingale in this theorem should be $(W_{T_t})_{t\geq 0}$ and the quasidiffusion $X$ is obtained by killing it at $\zeta$.  
\item[(2)] Note that $0\notin K$ (resp.  $1\notin K$) amounts to $l_0=l=-\infty$ (resp.  $r=r_0=\infty$).  
Suppose $0\in K$ (resp.  $1\in K$).  Then $\ell(\cdot, b_0)\equiv 0$ (resp.  $\ell(\cdot, a_1)\equiv 0$) if and only if $b_0=l_0>-\infty$ (resp.  $a_1=r_0<\infty$).  Clearly $l_0>-\infty$ (resp. $r_0<\infty$) implies $b_0=l_0$ (resp.  $a_1=r_0$).   Particularly,  $X$ is a local martingale,  if and only if 
\[
\text{neither }-\infty=l_0<l \text{ nor }r<r_0=\infty. 
\]
 For $k\in K'$,  $\ell(\cdot, a_k)\equiv 0$ (resp.  $\ell(\cdot,b_k)\equiv 0$),  if and only if $-\infty<l_0=a_k<b_k=l$ (resp.  $r=a_k<b_k=r_0<\infty$);   this is impossible if \text{(QK)} is assumed.  
\item[(3)] Recall that in \eqref{eq:51} we use another index set $\{k: k\geq 1\}$ to decompose the open set $[l,r]\setminus \overline{E}$.  When \text{(QK)} is assumed and $E=E_m$,  the sequence of open intervals in \eqref{eq:51} is identified with $\{I_k, k\in K'\}$.  In addition,  $0\in K$ (resp.  $1\in K$) if and only if $l$ (resp.  $r$) is finite.  Meanwhile $b_0=l$  (resp.  $a_1=r$).  
\end{itemize}
\end{remark}

\subsection{Semimartingale local times}

Now we turn to prepare the It\^o-Tanaka-Meyer formula (see,  e.g.,  \cite[Chapter VIII,  (29)]{DM82} and \cite[Theorem~9.46]{HWY92}) for the semimartingale $X$,  which will be used to prove the main result of \S\ref{SEC7}.  %The crucial element is the \emph{semimartingale local time} $L(t,x)$ at each $x\in\bR$.  
More precisely,  let $F$ be a convex function on $\bR$,  $F'_-$ be its left derivative and $F''$ be the Radon measure corresponding to the second derivative of $F$ in the sense of distribution.  Then the It\^o-Tanaka-Meyer formula states that $F(X)$ is also a semimartingale and admits the representation:
\begin{equation}\label{eq:516}
\begin{aligned}
	F(X_t)&=F(X_0)+\int_0^t F'_-(X_{s-})dX_s \\
	&+\sum_{0<s\leq t}\left[F(X_s)-F(X_{s-})-F'_-(X_{s-})\Delta X_s \right]+\frac{1}{2}\int_{\bR}L(t,x)F''(dx),
\end{aligned}\end{equation}
where $\Delta X_s:=X_s-X_{s-}$ and for any $x\in \bR$, $L(\cdot, x)$,  called the \emph{semimartingale local time} of $X$ at $x$,  is a continuous adapted increasing process with $L(0,x)=0$.  The expression of semimartingale local times is formulated in the following lemma.  For convenience we make the convention $L(\cdot,  \pm \infty)\equiv 0$.  

\begin{lemma}\label{LM65}
The semimartingale local time of $X$ at $x\in \bR$ is
\[
	L(t,x)=\left\lbrace
	\begin{aligned}
	&2\ell(t,x),\quad x\in E_m\setminus \{a_k: k\in K\};\\
	&0,\qquad\quad\;\; \text{otherwise}. 
	\end{aligned}
	\right.
\]
\end{lemma}
\begin{proof}
Note that $W_{T_{t-}}\in \tilde{E}_m$ and thus \cite[Theorem~9.44]{HWY92} yields $L(\cdot, x)\equiv 0$ for $x\notin \tilde{E}_m$.  For $x=l_0>-\infty$,  we assert that if $X_{t-}=l_0$ then $t\geq \zeta$ and $X_s=l_0$ for any $s\geq \zeta$.  This fact tells us that,  in view of \cite[Theorem~9.44]{HWY92},  $dL(\cdot, l_0)$ does not charge $(0,\zeta)\cup (\zeta, \infty)$,  as leads to $L(\cdot, l_0)\equiv 0$.  To prove the assertion,  note that $T_{t-}\leq T_t\leq \tau=\inf\{s>0:W_s\notin (l_0,r_0)\}$ by Lemma~\ref{LM515}.  Hence $W_{T_{t-}}=X_{t-}=l_0$ implies $T_{t-}=T_t=\tau<\infty$ and $W_\tau=l_0$.   We have shown in the proof of Lemma~\ref{LM515} that $T_\zeta=\tau$.  Since $T_s$ is strictly increasing before $\zeta$,  it follows that $t\geq \zeta$.  Particularly $X_s=l_0$ for any $s\geq \zeta$.  We eventually obtain $L(\cdot, l_0)\equiv 0$.  Analogously for $x=r_0<\infty$ we have $L(\cdot,  r_0)\equiv 0$.  

Next consider $x=a_k$ for some $k\in K$.  When $k=1\in K$,  it holds that $a_1=r_0<\infty$ or $a_1=r<r_0=\infty$.  In the former case we have already proved that $L(\cdot, a_1)\equiv 0$.  In the latter case $X_t\leq a_1$ and hence $X_{t-}\leq a_1$ for any $t\geq 0$.  Using Tanaka-Meyer formula (see,  e.g.,  \cite[Theorem~9.43]{HWY92}),   we get $L(\cdot, a_1)\equiv 0$.  Now let $k\in K'$.  Set $\varphi(t):=\int_0^t 1_{I_k}(u)du$, which is the difference of two convex functions.  Applying \eqref{eq:516} to $\varphi$,  we have
\begin{equation}\label{eq:517}
\begin{aligned}
	\varphi&(X_t)-\varphi(X_0)=\int_0^t 1_{(a_k,b_k]}(X_{s-})dX_s \\ &+\sum_{0<s\leq t}\left[\varphi(X_s)-\varphi(X_{s-})-1_{(a_k,b_k]}(X_{s-})\Delta X_s \right] +\frac{1}{2}L(t,a_k)-\frac{1}{2}L(t,b_k).  
\end{aligned}\end{equation}
It is easy to verify that $\varphi(X_t)-\varphi(X_0)=(b_k-a_k)(A_k(t)-B_k(t))$;  see also \cite[page 243]{BK87}.  Since $\ell(\cdot,  b_0)$ increases only at $t$ such that $X_t=b_0$ or $X_{t-}=b_0$,  it follows that
\begin{equation}\label{eq:518}
\int_0^t 1_{(a_k,b_k]}(X_{s-})d\ell(s, b_0)\leq 1_{\{a_k=b_0\}}\int_0^t 1_{\{X_{s-}=b_k, X_s=a_k\}}d\ell(s,b_0).
\end{equation}
Note that $\{0<s\leq t: X_{s-}=b_k, X_s=a_k\}$ contains only finite elements and $\ell(\cdot, b_0)$ is continuous.  Hence the left hand side of \eqref{eq:518} is equal to $0$.  Analogously one can compute that if $b_k=a_1$ then
\[
	\int_0^t 1_{(a_k,b_k]}(X_{s-})d\ell(s, a_1)=\ell(t,a_1);
\]
otherwise it is equal to $0$.  
If $b_k=a_p$ for some $p\in K'$,  the second term on the right side of \eqref{eq:517} is equal to 
\[
	(b_k-a_k)A_k(t)-(b_p-a_p)A_p(t);
\]
otherwise it is equal to $(b_k-a_k)A_k(t)$.  Using these computations and applying Theorem~\ref{THM517},  we can obtain that if $b_k=a_1$ or $a_p$,  then $\frac{1}{2}L(\cdot,a_k)-\frac{1}{2}L(\cdot, b_k)$ is a local martingale;  otherwise $\ell(\cdot, b_k)+\frac{1}{2}L(\cdot,a_k)-\frac{1}{2}L(\cdot, b_k)$ is a local martingale.  Hence either $L(\cdot, a_k)-L(\cdot, b_k)\equiv 0$ or $\ell(\cdot, b_k)+\frac{1}{2}L(\cdot,a_k)-\frac{1}{2}L(\cdot, b_k)\equiv 0$.  Note that $dL(\cdot, x)$ does not charge $\{t: X_{t-}\neq x\}$ for any $x\in \bR$;  see,  e.g.,  \cite[Theorem~9.44]{HWY92}.  Particularly,  the positive measures $dL(\cdot, a_k)$ and $dL(\cdot, b_k)$ on $[0,\infty)$ are singular.  
Therefore one can easily conclude that $L(\cdot, a_k)\equiv 0$.  In addition,  we also get that if $b_k\neq a_1, a_p$,  then $L(\cdot, b_k)=2\ell(\cdot, b_k)$.  

Thirdly consider $x\in E_m$ such that there are $y_0>y>x$ such that $[x,y]\subset [x, y_0)\subset E_m$.  Put $\varphi(t):=\int_0^t 1_{(x,y]}(u)du$.  Applying \eqref{eq:516} to $\varphi$ we have
\begin{equation}\label{eq:519}
	\varphi(X_t)-\varphi(X_0)=\int_0^t 1_{(x,y]}(X_{s-})dX_s +\frac{1}{2}L(t,x)-\frac{1}{2}L(t,y).  
\end{equation}
The first term on the right hand side is clearly a local martingale.  
On the other hand,  applying Tanaka's formula for Brownian motion we get
\[
	\varphi(W_t)-\varphi(W_0)=\int_0^t 1_{(x,y]}(W_s)dW_s+\ell^W(t,x)-\ell^W(t,y).  
\]
Write $U_t:=\int_0^t1_{(x,y]}(W_s)dW_s$.  
It follows that
\begin{equation}\label{eq:520}
	\varphi(X_t)-\varphi(X_0)=U_{T_t}+\ell(t,x)-\ell(t,y).  
\end{equation}
Note that $W_u\notin E_m$ for $u\in (T_{t-},T_t)$.  The bracket $\langle U\rangle$ and thus also $U$ is constant on $(T_{t-},T_t)$.  This implies the $T$-continuity of $U$ in the sense of \cite[Chapter V,  Definition~1.3]{RY99}.  In view of \cite[Chapter V, Proposition~1.5]{RY99},   we obtain that $U_{T_t}$ is a continuous local martingale.  As a result,  \eqref{eq:519} and \eqref{eq:520} yield
\[
	\frac{1}{2}L(t,x)-\frac{1}{2}L(t,y)=\ell(t,x)-\ell(t,y).  
\]
Note that $dL(\cdot, x)$ and $dL(\cdot, y)$ are singular.  So are $d\ell(\cdot, x)$ and $d\ell(\cdot, y)$.  Particularly we must have $L(t,x)=2\ell(t,x)$ and $L(t,y)=2\ell(t,y)$.  

Finally let $x\in E_m$ such that there exists a sequence of intervals in $\{I_k: k\in K'\}$ decreasing to $x$.  Without loss of generality we assume that $x<b_2$ and $(b_2, a_1)\cap E_m\neq \emptyset$.  It may appear that the intervals decreasing to $x$ are all ended by points in $\{a_p: p\in K \}$.  We ignore this case at first and suppose that $b_2\notin \{a_p: p\in K\}$.  Put $\varphi(t):=\int_0^t 1_{(x, b_2]}(u)du$.  Applying \eqref{eq:516} to $\varphi$ we get
\begin{equation}\label{eq:521}
	\varphi(X_t)-\varphi(X_0)=\int_0^t 1_{(x,b_2]}(X_{s-})dX_s+\frac{1}{2}L(t,x)-\frac{1}{2}L(t,b_2),
\end{equation}
where the first term on the right hand side is a local martingale due to $(b_2,a_1)\cap E_m\neq \emptyset$.  
Mimicking \eqref{eq:520} we also have
\begin{equation}\label{eq:522}
	\varphi(X_t)-\varphi(X_0)=\int_0^{T_t}1_{(x,b_2]}(W_s)dW_s+\ell(t,x)-\ell(t,b_2).  
\end{equation}
Put $\varphi_k(t):=\int_0^t 1_{(a_k,b_k]}(u)du$ for $k\in K$ such that $I_k\subset (x, b_2)$.  Denote by $K''$ the set of all these $k$.  Then 
\begin{equation}\label{eq:523}
\varphi_k(X_t)-\varphi_k(X_0)=(b_k-a_k)(A_k(t)-B_k(t))
\end{equation}
and on account of Tanaka's formula for Brownian motion, 
\begin{equation}\label{eq:524}
	\varphi_k(X_t)-\varphi_k(X_0)=\int_0^{T_t}1_{(a_k,b_k]}(W_s)dW_s + \ell(t,a_k)-\ell(t,b_k).  
\end{equation}
Repeating the argument in the previous step,  one can conclude that
\[
	t\mapsto \int_0^{T_t} 1_{(x,b_2]\setminus \cup_{k\in K''}(a_k,b_k]}(W_s)dW_s
\]
is a continuous local martingale.  Therefore \eqref{eq:521}, \eqref{eq:522}, \eqref{eq:523} and \eqref{eq:524} tell us
\[
\frac{1}{2}L(t,x)-\frac{1}{2}L(t,b_2)=\ell(t,x)-\ell(t,b_2).  
\]
Since we have obtained $L(t,b_2)=2\ell(t,b_2)$ when treating the case $x=a_k$ for $k\in K$,  it follows that $L(t,x)=2\ell(t,x)$.  At last we consider the ignored case with $b_2=a_k$ for some $k\in K'$.  Currently there is an additional term $-(b_k-a_k)A_k(t)$ in \eqref{eq:521}.  The above argument leads to $\frac{1}{2}L(t,x)-\frac{1}{2}L(t,b_2)=\ell(t,x)$,  while $L(t,b_2)=L(t,a_k)\equiv 0$ as proved in the second step.  Therefore we still have $L(t,x)=2\ell(t,x)$.  That completes the proof. 
\end{proof}

\section{Correspondence between skip-free Hunt processes and quasidiffusions}\label{SEC7}

Two Markov processes are called \emph{equivalent} if their transition functions coincide.   
The main purpose of this section is to prove the equivalence between SFHs on natural scale and quasidiffusions with the property (QK).  Recall that the state spaces of SFH on its natural scale and quasidiffusion are both nearly closed subsets of $\bR$ (not of $\overline{\bR}$!); see Remarks~\ref{RM48} and Lemma~\ref{LM51}.  Fix $E\in \mathscr K$ ended by $l$ and $r$,  and we add a ceremony $\partial$  attached to $E$ by the same way as in \S\ref{SEC51}.  
The main result is as follows.

\begin{theorem}\label{THM71}
Let $E\in \mathscr K$ and $\mu$ be a fully supported positive Radon measure on $E$.  The following Markov processes are equivalent:
\begin{itemize}
%\item[\rm(i)] The image regularized Markov process $\widehat{X}$ associated to $(I,\bs,\fm)$ whose image regularization $(\widehat{I},  \widehat{\bs}, \widehat{\fm})$ satisfies $\widehat{I}=E$ and $\widehat{\fm}=\mu$;
\item[\rm(i)] An SFH on $E$ on its natural scale whose speed measure is equal to $\mu$;
\item[\rm(ii)] A quasidiffusion with speed measure $m$ such that $E_m=E$,  $m|_{E}=2\mu$ and \text{(QK)} holds for $m$.
\end{itemize}
\end{theorem}
\begin{remark}\label{RM42}
This result,  together with Remark~\ref{RM48},  readily implies that for any SFH $X$ on $E\in \overline{\mathscr K}$ with scale function $\bs$,  $\tilde{X}:=\bs(X)$ is a quasidiffusion on $\bs(E)\in \mathscr K$.  
\end{remark}

%Three stochastic models in one dimension have been argued in previous sections:
%\begin{itemize}
%\item[(i)]  The first one is the canonical regularized Markov process $\widehat X$ associated to $(I,\bs,\fm)$ obtained in Theorem~\ref{THM6}.  Its speed measure is $\widehat{\fm}$.  
%\item[(ii)] The second is an SFH,  which can be transformed into another SFH on its natural scale by virtue of the scale function obtained in Theorem~\ref{THM58}.  Its speed measure $\mu$ is defined for various cases in Definition~\ref{DEF514}. 
%\item[(iii)]  The third is a quasidiffusion with speed measure $m$,  as reviewed in \S\ref{SEC6}. 
%,  which is always on its natural scale in the sense that it corresponds to the generalized second order differential operator $\frac{d^2}{dmdx}$,  where $m$ is the speed measure of quasidiffusion. 
%\end{itemize}
% In this section,  we will prove that when are on natural scales and have identifying speed measures,  they are all 

\subsection{Proof of Theorem~\ref{THM71}}\label{SEC41}

 Since involved and long,  the proof of Theorem~\ref{THM71} will be divided into several parts.  As a byproduct,  we will prove that both Markov processes in Theorem~\ref{THM71} are also equivalent to another time-changed Brownian motion introduced as below. 
 
Put an open interval
\begin{equation}\label{eq:41-2}
J:=(l^0, r^0),
\end{equation}
where $l^0:=l$ if $l\notin E$ and $l^0:=-\infty$ otherwise,  $r^0:=r$ if $r\notin E$ and $r^0:=\infty$ otherwise. 
Let $W^0$ be the absorbing Brownian motion on $J$ associated with the Dirichlet form $(\frac{1}{2}\mathbf{D}, H^1_0(J))$ where $H^1_0(J)$ is the closure of the family of smooth functions with compact support on $J$ in the Sobolev space of order $1$ over $J$ and for $f,g\in H^1_0(J)$,  $\frac{1}{2}\mathbf{D}(f,g)=\frac{1}{2}\int_J f'(x)g'(x)dx$; see,  e.g.,  \cite[Example~3.5.7]{CF12}.  The lifetime of $W^0$ is denoted by $\zeta^0$.  Viewed as a measure on $J$,  $\mu$ is a Radon smooth measure with respect to $W^0$.  In addition,  the quasi support of $\mu$ is identified with its topological support $E$,  which is a closed subset of $J$.  Denote by $\widehat{S}=(\widehat S_t)_{t\geq 0}$ the PCAF of $\mu$ with respect to $W^0$.  Set 
\[
\widehat T_t:=\left\lbrace 
	\begin{aligned}
		&\inf\{u: \widehat S_u>t\},\quad t<\widehat S_{\zeta^0-},\\
		&\infty,\qquad\qquad\qquad\;\; t\geq \widehat S_{\zeta^0-}  
	\end{aligned}
\right.
\]
and 
\[
\widehat{X}_t:=W^0_{\widehat T_t}, \;t\geq 0,\quad \widehat{\zeta}:=\widehat S_{\zeta^0-}.
\]
Then $\widehat{X}$ is a right process on $E$ with lifetime $\widehat{\zeta}$,  called the \emph{time-changed Brownian motion} with speed measure $\mu$;  see,  e.g.,  \cite[Theorem~A.3.11]{CF12}.   In view of \cite[Corollary~5.2.10]{CF12},  $\widehat X$ is symmetric with respect to $\mu$ and its associated Dirichlet form on $L^2(E,\mu)$ is regular.  Particularly $\widehat{X}$ is a Hunt process on $E$.  

%Clearly the quasi support of $\widehat{\fm}$ is identified with its topological support $\widehat{I}$.   Denote the PCAF of $\widehat{\fm}$ with respect to $\widehat{B}$ by $A=(A_t)_{t\geq 0}$.  Set 

%We first prove that the Dirichlet form $(\widehat{\sE},\widehat{\sF})$ of $\widehat{X}$ is the trace Dirichlet form of $(\frac{1}{2}\mathbf{D}, H^1_0(J))$ on $L^2(E,\mu)$.  In other words,  $\widehat{X}$ is the time-changed process of $W$ by the PCAF corresponding to $\mu$ (cf.  Theorem~\ref{THM19}).

\subsubsection{Quasidiffusion and time-changed Brownian motion}\label{SEC711}
%Denote by $X$ the quasidiffusion in (iii).  

%\begin{lemma}
%Let $\widehat{X}$ be in (i).  Then its Dirichlet form $(\widehat{\sE},\widehat{\sF})$ is the trace Dirichlet form of $(\frac{1}{2}\mathbf{D}, H^1_0(J))$ on $L^2(E,\mu)$.  
%\end{lemma}
%\begin{proof}
%Adopt the same notations as in the proof of Theorem~\ref{THM19}.
%Further let $(\sA,\sG)$ be the trace Dirichlet form of $(\frac{1}{2}\mathbf{D}, H^1_0(J))$ on $L^2(E,\mu)$.  Denote by $H^1_e(J)$ the extended Dirichlet space of $(\frac{1}{2}\mathbf{D}, H_0^1(J))$; see \eqref{eq:320-2}.  Note that $H^1_e(J)|_{\widehat{I}}=H^1_e(\widehat{J})|_{\widehat{I}}$.  It follows from \eqref{eq:17} that $\sG=H^1_0(J)|_{\widehat{I}}\cap L^2(\widehat{I},\widehat{\fm})=\widehat{\sF}$.  Repeating the argument in the proof of \cite[Theorem~2.1]{LY17},  we can still obtain that $\sA(\widehat\varphi,\widehat\varphi)=\widehat{\sE}(\widehat{\varphi},\widehat{\varphi})$ for $\widehat{\varphi}\in \sG$.  That completes the proof.  
%\end{proof}

We first show that the quasidiffusion is identified with the time-changed Brownian motion $\widehat{X}$. 
% of $W^0$ by the PCAF corresponding to $\mu$,  so that the equivalence between $X$ and $\widehat{X}$ is obtained.  

\begin{proof}[Proof of identification between quasidiffusion and $\widehat{X}$]
Let $X$ be the quasidiffusion in Theorem~\ref{THM71}.  
Adopt the same notations as in \S\ref{SEC6}.  The assumption \text{(QK)} implies that $l_0$ (resp.  $r_0$) in \eqref{eq:51-2} is identified with that defined in \eqref{eq:41-2}.  In addition,  $W^0$ is identified with the killed process of Brownian motion $W$ at $\tau_J:=\{t>0: W_t\notin J\}$.  Particularly $\tau_J=\tau$ where $\tau$ is defined in Lemma~\ref{LM515}. 
Note that for any $x\in J$,  $\ell^{W^0}(t,x):=\ell^W(t\wedge \tau, x), t\geq 0$,  is a PCAF of $W^0$ whose Revuz measure is $\frac{1}{2}\delta_x$; see, e.g.,  \cite[Chapter X,  Proposition~2.4]{RY99} and \cite[Proposition~4.1.10]{CF12}.  Hence the PCAF of $\mu$ with respect to $W^0$ is 
\[
	\widehat S_t=\frac{1}{2}\int_{E} \ell^{W^0}(t,x)\mu(dx)=\int_{E_m} \ell^{W^0}(t,x)m(dx),\quad t\geq 0.  
\]

%Clearly $\widehat{S}$ is a PCAF of $W^0$ corresponding to the Revuz measure $\mu$.  Define
%\[
%	\widehat{T}_t:=\left\lbrace
%	\begin{aligned}
%		&\inf\{u: \widehat{S}_u>t\},\quad \text{if }t<\widehat{S}_{\tau-},\\
%		&\infty,\qquad\qquad\qquad\;\;\, \text{if }t\geq \widehat{S}_{\tau-}. 
%	\end{aligned}
%	\right.
%\]
%We let 
%\[
%	\widehat{X}_t:=W^0_{\widehat{T}_t},\; t\geq 0,\quad \widehat{\zeta}:=\widehat{S}_{\tau-}.  
%\]
%Then $\widehat{X}$ is the time-changed process of $W^0$ by the PCAF $\widehat{S}$ and $\widehat{\zeta}$ is its lifetime; see,  e.g.,  \cite[Theorem~A.3.9]{CF12}.  

We turn to show the identification of $X$ and $\widehat{X}$.  In fact,  for $u<\tau$,  we have $\widehat{S}_u=S_u$.  It follows from Lemma~\ref{LM515} that $\zeta=S_{\tau-}=\widehat{S}_{\tau-}=\widehat{\zeta}$.  For $t<\zeta=\widehat{\zeta}$,  since $T_t=\inf\{u: S_u>t\}<\tau$ as obtained in Lemma~\ref{LM515},  it follows that
\[
	\widehat{T}_t=\inf\{u: \widehat{S}_u>t\}=\inf\{u<\tau: S_u>t\}=T_t.  
\]
As a result $\widehat{X}_t=W_{\widehat{T}_t\wedge \tau}=W_{T_t\wedge \tau}=W_{T_t}=X_t$ for $t<\zeta=\widehat{\zeta}$.   Therefore $X$ and $\widehat{X}$ are identified.  That completes the proof. 
\end{proof}
\begin{remark}\label{RM73}
The Dirichlet form of $\widehat{X}$ is expressed in,  e.g.,  \cite[Theorem~3.4]{L23}.  It contains no killing part.  
If \text{(QK)} is not assumed,  the argument in this proof still holds true by replacing $W^0$ with the absorbing Brownian motion on $(l_0,r_0)$,  where $l_0$ and $r_0$ are defined in \eqref{eq:51-2} instead.  In other words,  a general quasidiffusion $X$ without (QK) is still identified with the time-changed process of absorbing Brownian motion on $(l_0,r_0)$ by the PCAF corresponding to $\mu$.  Meanwhile its associated Dirichlet form enjoys an additional killing term
\[
	\frac{\widehat{f}(l)^2}{2|l_0-l|} \quad \left(\text{resp. }\frac{\widehat{f}(r)^2}{2|r_0-r|} \right),
\]
whenever $-\infty<l_0<l$ (resp.  $r<r_0<\infty$).  In the context of Markov processes, $X$ admits killing at $l$ or $r$ if $-\infty<l_0<l$ or $r<r_0<\infty$,  and the killing ratio is $1/(2|l-l_0|)$ or $1/(2|r-r_0|)$.  In the context of analysis, $X$ corresponds to the operator \eqref{eq:01} with Robin boundary condition at $l$ or $r$. %When $l_0\downarrow -\infty$ or $r_0\uparrow \infty$,  the killing disappears and when $l_0\uparrow l$ or $r\downarrow r_0$,  the killing occurs so frequently that $l$ or $r$ becomes an absorbing endpoint.  
\end{remark}

\subsubsection{Quasidiffusion as an SFH}\label{SEC712}

As noted in \S\ref{SEC6},  the quasidiffusion $X$ is a standard process on $E$ satisfying (SF).  Due to the argument in \S\ref{SEC711},  it is identified with $\widehat{X}$ and hence actually a Hunt process on $E$.  In addition,  applying \cite[Theorems~5.2.8~(2) and Theorem~3.5.6~(1)]{CF12},  we can obtain (SR) for $X$.   %On account of \cite[Theorem~3.1.10]{CF12},  a semipolar set for $X$ must be empty.  Thus Hunt's hypothesis holds for $X$.  
Since $\widehat{X}$ admits no killing inside as mentioned in Remark~\ref{RM73},  it follows that (SK) holds true for $X$.  Therefore we can eventually conclude that $X$ is an SFH on $E$.  

The fact that $X$,  as an SFH,  is on its natural scale can be seen by virtue of \cite[(6,7)]{K86}.  In what follows we prove that the speed measure of $X$, as an SFH, is equal to $\mu$.  

%Before moving on we prepare a lemma.  
%\begin{lemma}
%Let $X$ be the quasidiffusion in (iii) and $\ell(t,x)$ be its Markov local time.  Then $\ell(\cdot,x)$ is a PCAF of $X$ whose Revuz measure is $\frac{1}{2}\delta_x$ for $x\in E$.  
%\end{lemma}
%\begin{proof}
%Using \cite[Proposition~A.3.8~(2)]{CF12} and \cite[Chapter X, Proposition~1.2]{RY99},  one can easily verify that $\ell(\cdot, x)$ is a PCAF of $X$.  Denote by $\nu_x$ its Revuz measure,  i.e.  for any positive Borel measurable function $g\in L^1(E,\mu)$, 
%\[
%\nu_x(g)=\lim_{t\downarrow 0} \frac{1}{t}\int_{E}\mu(dy) \mathbf{E}_y \left(\int_0^t g(X_s)d\ell(s,x) \right). 
%\]
%Take another positive and bounded Borel measurable function $f$ on $E$.  It follows from \eqref{eq:61} that
%\[
%A	
%\]
%\end{proof}

\begin{proof}
Adopt the same notations as in \S\ref{SEC53}. 

Let us first treat the case $l,r\notin E$.  Take $a,b\in E$ with $a<b$.  Set $T:=T_a\wedge T_b$ and let $h:=h_{a,b}$ be defined as in \S\ref{SEC531}. 
We assert that for any $x\in E$ with $a<x<b$,  
\begin{equation}\label{eq:72}
M_t:=h(X_{t\wedge T})+t\wedge T=\mathbf{E}_x(T|\sF_t)
\end{equation} 
is a $\mathbf{P}_x$-martingale.  In fact,  $T<\infty$,  $\mathbf{P}_x$-a.s.  and we have
\[
M_t=\mathbf{E}_x \left(T\circ \theta_{t\wedge T}|\sF_{t\wedge T} \right)+t\wedge T=\mathbf{E}_x (T|\sF_{t\wedge T})=\mathbf{E}_x(T|\sF_t).  
\]	
Hence \eqref{eq:72} holds.  Note that $h$ is concave as proved in Lemma~\ref{LM511}.  Applying It\^o-Tanaka-Meyer formula \eqref{eq:516} to $h$,  we get
\begin{equation}\label{eq:71}
\begin{aligned}
h&(X_{t\wedge T})-h(X_0)=\int_0^{t\wedge T}h'_-(X_{s-})dX_s  \\
&+\sum_{0<s\leq t\wedge T} \left[h(X_s)-h(X_{s-})-h'_-(X_{s-})\Delta X_s\right]+\frac{1}{2}\int_\bR L(t\wedge T,  x)h''(dx).  
\end{aligned}
\end{equation}
In view of Remark~\ref{RM64}~(2),  the first term on the right hand side is a local martingale.  Note that $\Delta X_s\neq 0$ if and only if $X_s=a_k, X_{s-}=b_k$ or $X_s=b_k,  X_{s-}=a_k$ for some $k\in K'$,  where $K'$ is given before Theorem~\ref{THM517}.  Using the expression of $h$,  one yields that the second term on the right hand side of \eqref{eq:71} is equal to 
\[
\sum_{k\in K'}(h'_+(a_k)-h'_-(a_k))(b_k-a_k) A_k(t\wedge T).  
\]
By virtue of Theorem~\ref{THM517} and Lemma~\ref{LM65},  it follows from \eqref{eq:72} and \eqref{eq:71} that
\begin{equation}\label{eq:73}
	t\mapsto \int_{E} \ell(t\wedge T,x)h''(dx)+t\wedge T
\end{equation}
is a local martingale.   On account of $T_t<\tau$ for $t\leq T<\zeta$,  we have that $\ell(t\wedge T, x)\leq \ell^W(T_t, x)<\infty$ for $x\in (a,b)$ and $\ell(t\wedge T, x)=0$ for $x\notin (a,b)$.  Note that $-h''((a,b))<\infty$.  Hence \cite[Corollary~2.10]{RY99} yields that \eqref{eq:73} is also a continuous process locally of bounded variation and must be equal to $0$ for all $t\geq 0$.  Letting $a\downarrow l$ and $b\uparrow r$ we get
\begin{equation}\label{eq:74}
	-2\int_E \ell(t,x)\tilde{\mu}(dx)+t=0,\quad 0\leq t<\zeta,
\end{equation}
where $\tilde{\mu}$ defined as \eqref{eq:510} is the speed measure of $X$ as an SFH.  
Further let $t:=S_u<\zeta$ for $u<\tau$.  Since $W_v\notin E$ for $v\in (u, T_{t})$ if $u<T_t$,  it follows that $\ell(t,x)=\ell^W(T_t,x)=\ell^W(u,x)$.  Thus \eqref{eq:74} reads as
\begin{equation}\label{eq:75}
	2\int_E \ell^W(u\wedge \tau,x)\tilde{\mu}(dx)=S_u=2\int_E\ell^W(u\wedge \tau ,x)\mu(dx),\quad u\geq 0,
\end{equation}
where the second identity is due to \eqref{eq:61}.  Note that the first  (resp.  second) integration in \eqref{eq:75} is a PCAF of $W^0$ whose Revuz measure is $\tilde{\mu}$ (resp.  $\mu$).  Therefore the uniqueness of Revuz correspondence leads to $\mu=\tilde{\mu}$.  

Next consider the case $l\in E,  r\notin E$.  Take $l< b\in E$ and set $T:=T_b$.  Let $h:=h_b$ be defined as in \S\ref{SEC532}.  Then \eqref{eq:72} still holds for $x\in E$ with $x<b$ and particularly,  $M_t$ is a $\mathbf{P}_x$-martingale.   Applying the It\^o-Tanaka-Meyer formula to this $h$ we still get \eqref{eq:71}.  But now the first term on the right hand side of \eqref{eq:71} is the sum of a local martingale and
\[
	\int_0^{t\wedge T}h'_-(X_{s-})d\ell(s,l).  
\]
Note that $\ell(\cdot, l)$ increases only at $X_{s}=l$ or $X_{s-}=l$.  If $X_{s}=l$ but $X_{s-}\neq l$,  then $l=a_k<b_k=X_{s-}$ for some $k\in K'$ due to (SF).  The number of these times is finite before $t\wedge T$.  Since $\ell(\cdot, l)$ is continuous,  it follows that
\[
	\int_0^{t\wedge T}h'_-(X_{s-})d\ell(s,l)=h'_-(l)\ell(t\wedge T,l)=0.
\]
Repeating the argument for the first case,  we can still obtain \eqref{eq:75} for $\tilde{\mu}$ defined as \eqref{eq:513}.  Eventually $\tilde{\mu}=\mu$ can be concluded.  

Another two cases $l\notin E, r\in E$ and $l,r\in E$ can be argued analogously.  The proof is completed.  
\end{proof}

\subsubsection{Skip-free Hunt process and time-changed Brownian motion}

Finally let $X$ be an SFH on $E$ on its natural scale whose speed measure is equal to $\mu$.  We are to prove that $X$ is equivalent to the time-changed Brownian motion $\widehat{X}$.

\begin{proof}
As proved in \S\ref{SEC711} and \S\ref{SEC712},  $\widehat{X}$ is also an SFH on $E$ on its natural scale whose lifetime is $\widehat{\zeta}$.  Set $\widehat{T}_x:=\inf\{t>0:\widehat{X}_t=x\}$ for any $x\in E$.  

For any non-empty compact set $K\subset E_\partial$,  define the hitting distributions for $X$ and $\widehat{X}$ as follows: For $A\subset K$, 
\[
	P_K(x,A):=\mathbf{P}_x(X_{T_K}\in A, T_K<\zeta),\quad \widehat{P}_K(x,A):=\mathbf{P}_x(\widehat{X}_{\widehat{T}_K}\in A,  \widehat{T}_K<\widehat{\zeta}),
\]
where $T_K:=\inf\{t>0:X_t\in K\}$ and $\widehat{T}_K:=\inf\{t>0:\widehat{X}_t\in K\}$.  We assert that 
\begin{equation}\label{eq:77}
	P_K(x,\cdot)=\widehat{P}_K(x,\cdot),\quad \forall x\in E.
\end{equation}
Only the case $l\in E,  r\notin E$ will be treated and the others can be argued analogously.  Note that $\partial$ is identified with $r$,  and in view of Lemma~\ref{LM54}~(4), 
\begin{equation*}
\zeta=T_r:=\lim_{b\uparrow r}T_b,\quad \widehat{\zeta}=\widehat{T}_r:=\lim_{b\uparrow r}\widehat{T}_b.  
\end{equation*}
%Corollary~\ref{COR59},  $r=\infty$ implies that $T_r=\zeta=\infty$ and $\widehat{T}_r=\widehat{\zeta}=\infty$,  where $\widehat{T}_r:=\lim_{b\uparrow r}\widehat{T}_b$.  
 When $x\in K$,  $T_K\leq T_x=0$ and $\widehat{T}_K\leq \widehat{T}_x=0$.  Hence $P_K(x,\cdot)=\widehat{P}_K(x,\cdot)=\delta_x$.  When $x\notin K$,  set $K_-:=\{y\in K: y<x\}$ and $K_+:=\{y\in K: y>x\}$.  We verify \eqref{eq:77} for different cases as follows:
\begin{itemize}
\item[(1)] $K_-\neq \emptyset,  K_+=\emptyset$: Set $a:=\sup\{y: y\in K_-\}$.  On account of Lemma~\ref{LM54}~(2),  $T_K=T_a$ and $\widehat{T}_K=\widehat{T}_a$.  Then both $P_K(x,\cdot)$ and $\widehat{P}_K(x,\cdot)$ are concentrated on $\{a\}$.  It follows from Theorem~\ref{THM58} that
\[
P_K(x,\{a\})=\mathbf{P}_x(T_a<\zeta)=\lim_{b\uparrow r}\mathbf{P}_x(T_a<T_b)=\lim_{b\uparrow r}\frac{b-x}{b-a},
\]
which is equal to $(r-x)/(r-a)$ if $r<\infty$ and $1$ if $r=\infty$.  Analogously one can also compute
\[
	\widehat P_K(x,\{a\})=\lim_{b\uparrow r}\frac{b-x}{b-a}.
\]
Thus \eqref{eq:77} is verified.
\item[(2)] $K_-=\emptyset,  K_+\neq \emptyset$: Set $b:=\inf\{y: y\in K_+\}$.  If $b=r$,  then $T_K=\zeta$ and $\widehat{T}_K=\widehat{\zeta}$.  Hence $P_K(x,\cdot)=\widehat{P}_K(x,\cdot)\equiv 0$.  If $b<r$,  then $T_K=T_b$ and $\widehat{T}_K=T_b$.  Both $P_K(x,\cdot)$ and $\widehat{P}_K(x,\cdot)$ are concentrated on $\{b\}$.  Note that $\mathbf{P}_x(T_b<\infty)=\mathbf{P}_x(\widehat{T}_b<\infty)=1$;  see \S\ref{SEC532}.  One can easily obtain that $P_K(x, \{b\})=\widehat{P}_K(x,\{b\})=1$.  Hence \eqref{eq:77} holds true.
\item[(3)] $K_-,K_+\neq \emptyset$: Set $a:=\sup\{y: y\in K_-\}$ and $b:=\inf\{y: y\in K_+\}$.   If $b=r$,  then \eqref{eq:77} has been verified in the case $K_-\neq \emptyset,  K_+=\emptyset$.  Now suppose $b<r$.  Then both $P_K(x,\cdot)$ and $\widehat{P}_K(x,\cdot)$ are concentrated on $\{a,  b\}$.  Applying Theorem~\ref{THM58},  we get
$P_K(x,\{a\})=\mathbf{P}_x(T_a<T_b)=\frac{b-x}{b-a}$ and $P_K(x,\{b\})=\frac{x-a}{b-a}$.  Similarly $\widehat P_K(x,\{a\})=\frac{b-x}{b-a}$ and $\widehat P_K(x,\{b\})=\frac{x-a}{b-a}$.  Therefore \eqref{eq:77} is verified.  
\end{itemize}

With \eqref{eq:77} at hand,  we apply the Blumenthal-Getoor-McKean theorem (see,  e.g.,  \cite[Chapter V,  Theorem~5.1]{BG68}) to obtain a CAF $A=(A_t)_{t\geq 0}$ of $\widehat{X}$ which is strictly increasing and finite on $[0,\widehat{\zeta})$ such that $X$ is equivalent to the time-changed process of $\widehat X$ by $A$.  In view of \cite[Chapter V, Theorem~2.1]{BG68} (the symmetric measure $\mu$ of $\widehat{X}$ is clearly a reference measure in the sense of \cite[Chapter V, Definition~1.1]{BG68} due to \cite[Chapter V, Proposition~1.2]{BG68}),  $A$ is a perfect CAF in the sense of \cite[Chapter IV,  Definition~1.3]{BG68}.  Hence $A$ is a PCAF of $\widehat{X}$ in the sense of,  e.g.,  \cite[Definition~A.3.1]{CF12}.  Denote the Revuz measure of $A$ by $\widehat{\mu}$.  Since $A$ is strictly increasing and the polar set for $\widehat{X}$ must be empty (see,  e.g.,  \cite[Corollary~3.4]{L23}),  it follows that the quasi support of $\widehat{\mu}$ is identified with its topological support $\text{supp}[\widehat{\mu}]=E$.  

We prove that $\widehat{\mu}$ is a Radon measure on $E$.  In fact,  there is an $\widehat{\sE}$-nest $\{\widehat{F}_n:n\geq 1\}$ such that $\widehat{\mu}(\widehat{F}_n)<\infty$.  For any compact set $K\subset E$,   it follows from \cite[Lemma~3.8]{L23} that $K\subset \widehat{F}_n$ for some $n$.  Particularly,  $\widehat{\mu}(K)<\infty$.   

Now applying \cite[Corollary~5.2.12]{CF12},  $X$,  as a time-changed process of $\widehat{X}$ by $A$,  is also a time-changed Brownian motion by the PCAF corresponding to $\widehat{\mu}$.  Repeating the argument in \S\ref{SEC712},  we can conclude that its speed measure in the sense of SFH is also equal to $\widehat{\mu}$.  Therefore $\widehat{\mu}=\mu$ and particularly,  $X$ is equivalent to $\widehat{X}$.  That completes the proof.  
\end{proof}

\subsection{Consequences}

As a corollary of Theorem~\ref{THM71},  we present the correspondence between an SFH and its scale function and speed measure.  The uniqueness in this corollary means that all processes satisfying the desirable conditions are equivalent.  

\begin{corollary}\label{COR64}
\begin{itemize}
\item[(1)] Let $E\in \mathscr K$ and $\mu$ be a fully supported positive Radon measure on $E$.  Then there exists a unique SFH on $E$ on its natural scale whose speed measure is $\mu$.  
\item[(2)] Let $E\in \overline{\mathscr K}$,  $\mu$ be a fully supported positive Radon measure on $E$ and $\bs$ be a strictly increasing and continuous real valued function on $E$.  Then there exists a unique SFH on $E$ with scale function $\bs$ and speed measure $\mu$. 
\end{itemize}
\end{corollary}
\begin{proof}
To prove the first assertion,  note that a time-changed Brownian motion is uniquely determined by its speed measure,  because so is its associated Dirichlet form.  Hence the uniqueness holds true by means of Theorem~\ref{THM71}.  It suffices to show the existence.  In view of the equivalence between SFH and quasidiffusion,  we only need to find a function $m$ as in \S\ref{SEC6} such that $E_m=E,  m|_E=2\mu$ and \text{(QK)} holds for $m$.  To accomplish this,  assume without loss of generality that $0\in (l,r)$ where $l$ and $r$ are endpoints of $E$.  Set
\begin{equation}\label{eq:78}
	m(x):=2\mu( (0,x]),\; x\in [0, r),\quad m(x):=-2\mu((x,0]),\; x\in [l,0).
\end{equation}
When $r<\infty$,  set
\[
\begin{aligned}
	&m(x)=\infty,  \;  x\geq r, \qquad\qquad  \text{if }r\notin E, \\
	&m(x)=2\mu((0,r]),  \;  x\geq r, \quad \text{if }r\in E. 
\end{aligned}
\]
When $l>-\infty$,  define
\[
\begin{aligned}
	&m(x)=-\infty,  \;  x< l,   \qquad\qquad  \text{if }l \notin E, \\
	&m(x)=-2\mu([l,0]),  \;  x< l, \quad\; \text{if }l\in E. 
\end{aligned}
\]
It is straightforward to verify that $m$ satisfies all these conditions.  Therefore the first assertion is concluded.  The second assertion can be obtained by applying the first one to $\tilde{E}:=\bs(E)$ and $\tilde{\mu}:=\mu\circ \bs^{-1}$.  That completes the proof.  
\end{proof}
\begin{remark}
The desirable function $m$ in this proof is unique (up to a constant) for a given pair $(E,\mu)$.  However if \text{(QK)} is not required,  then the uniqueness of $m$ does not hold.  To obtain $m$ uniquely without assuming \text{(QK)},  one need an additional function $k: \{l,r\}\cap E\rightarrow [0,\infty)$ to determine the killing ratio at $l\in E$ or $r\in E$.  More precisely,  let $m$ be defined as \eqref{eq:78} for $x\in [l,r)$.  When $r\in E$,  set
\[
	m(x):=2\mu((x,r]),\; x\in \left[r,  r+\frac{1}{2k(r)}\right),\quad m(x):=\infty,  \; x\geq r+\frac{1}{2k(r)}.
\]
When $r\notin E$ and $r<\infty$,  set $m(x):=\infty$ for $m\geq r$.  The definition of $m(x)$ for $x\leq l$ is analogical.  This $m$ leads to the unique quasidiffusion for the given triple $(E,\mu,k)$.  
\end{remark}

\section{Analytic characterization of skip-free Hunt processes}\label{SEC2}

 In this section we will turn to characterize SFHs by virtue of Dirichlet forms.  
As shown in Theorem~\ref{THM71},  an SFH is equivalent to a quasidiffusion up to a homeomorphism.  Particularly it is symmetric with respect to its speed measure,  and its associated Dirichlet form is regular; see \cite[Theorem~5.2.8]{CF12}.  In fact,  the speed measure is the unique symmetrizing measure of SFH (up to multiplicative constant) due to (SR) and \cite{YZ10}. 

From now on fix $E\in \overline{\mathscr K}$ ended by $l$ and $r$,  and $\partial$ is the ceremony as before. Write $[l,r]\setminus \overline{E}$ as \eqref{eq:51},  i.e. 
\begin{equation}\label{eq:50}
	[l,r]\setminus \overline{E}=\cup_{k\geq 1}(a_k,b_k),
\end{equation}
where $(a_k,b_k)$,  $k\geq 1$,  are disjoint open intervals.  
 Given a compact subset $K\subset E$,  a \emph{neighbour} of $K$ is an open set $U\subset E$ such that $K\subset U$.  A neighbour $U$ of $K$ is called a \emph{strong neighbour} if $a_k\in K$ (resp.  $b_k\in K$) implies $b_k\in U$ (resp.  $a_k\in U$) for each $k\geq 1$.  

Let $\mu$ be a fully supported Radon measure on $E$ and $(\sE,\sF)$ be a regular Dirichlet form on $L^2(E,\mu)$.  For $f\in L^2(E,\mu)$,  $\text{supp}[f]$ stands for the support of the measure $f\cdot \mu$.  The $1$-capacity with respect to $\mathscr E$ is denoted by $\text{Cap}$; see \cite[\S2.1]{FOT11}.  Recall that a $\mu$-measurable set $A\subset E$ is called \emph{invariant} if $T_t(1_Af)=1_A\cdot T_tf$,  $\mu$-a.e.  for any $f\in L^2(E,\mu)$ and $t>0$,  where $T_t$ is the $L^2$-semigroup corresponding to $(\sE,\sF)$.  In addition,  $(\sE,\sF)$ is called \emph{irreducible} if any invariant set $A$ satisfies either $\mu(A)=0$ or $\mu(E\setminus A)=0$.  

\begin{definition}\label{DEF51-2}
Let $(\sE,\sF)$ be a regular Dirichlet form on $L^2(E,\mu)$.  
\begin{itemize}
\item[(1)] $(\sE,\sF)$ is called \emph{strictly irreducible},  if it is irreducible and $\text{Cap}(\{l\})>0$ (resp.  $\text{Cap}(\{r\})>0$) whenever $l\in E$ (resp.  $r\in E$).  
\item[(2)] $(\sE,\sF)$ is called \emph{strongly local-like},  if $\sE(f,g)=0$ for any $f,g\in \sF$ such that $\text{supp}[f], \text{supp}[g]$ are compact and $g$ is constant on a strong neighbour of $\text{supp}[f]$.  
\end{itemize}
\end{definition}
\begin{remark}
%An irreducible $(\sE,\sF)$ may admit $\text{Cap}(\{l\})=0$ when $l\in E$;  see,  e.g.,  
Strongly local Dirichlet forms are always strongly local-like but not vice versa.   If $E$ is an interval,  then strong neighbours are just neighbours,  and particularly,  the strong local-like property coincides with the strong local property.  
\end{remark} 

%As stated in Corollary~\ref{COR310},  $X$ is transient if and only if either $-\infty<l\notin E$ or $\infty>r\notin E$.  Otherwise it is recurrent.  

Before stating the main result,  we need to prepare some notations.  
Let $\bs$ be a strictly increasing and continuous real valued function on $E$.  We extend $\bs$ to a strictly increasing and continuous function $\bar{\bs}$ on $[l,r]$ by letting
\[
	\bar{\bs}(x):=\bs(a_k)+\frac{\bs(b_k)-\bs(a_k)}{b_k-a_k}\cdot (x-a_k),\quad x\in (a_k,b_k), \;k\geq 1,
\]
and 
\[
\bar{\bs}(l):=\bs(l)=\lim_{x\downarrow l}\bs(x),\quad \bar{\bs}(r):=\bs(r)=\lim_{x\uparrow r}\bs(x).
\]	
Denote by $J:=\langle l,  r\rangle$ the interval ended by $l$ and $r$,  where $l\in J$ (resp.  $r\in J$) if and only if $l\in E$ (resp.  $r\in E$).  
Then the Lebesgue-Stietjes measure of $\bar{\bs}$, denoted by $\lambda_{\bar{\bs}}$ or $d\bar{\bs}$,  is Radon on $J$.  Further denote by $\lambda_\bs$ or $d\bs$ the restriction of $\lambda_{\bar{\bs}}$ to $E$.  
Given $f\in C(J)$,  $f\ll \bar{\bs}$ means that $f$ is absolutely continuous with respect to $\bar{\bs}$,  i.e.  there exists an absolutely continuous function $g$ on $(\bar{\bs}(l),\bar{\bs}(r))$ such that $f=g\circ \bar{\bs}$.  Meanwhile $df/d\bar{\bs}:=g'\circ \bar{\bs}$.   Put a family of continuous functions
\[
	H^1_{e,\bs}(E):=\left\{f|_{E}: f\in C(J),  f\ll \bar{\bs},  \frac{df}{d\bar{\bs}}\in L^2(J,  \lambda_{\bar{\bs}})\right\}.  
\]
In addition for $h=f|_E\in H^1_{e,\bs}(E)$, we make the convention
\[
\frac{dh}{d\bs}:=\frac{df}{d\bar{\bs}}\bigg|_E
\]
for convenience.  Note that every $h\in  H^1_{e,\bs}(E)$ admits a finite limit $h(j):=\lim_{x\rightarrow j}h(x)$ whenever $|\bs(j)|<\infty$ for $j=l$ or $r$ (even if $j\notin E$). 

\begin{theorem}\label{THM53}
Let $E\in \overline{\mathscr K}$ and $\mu$ be a fully supported Radon measure on $E$. 
The regular Dirichlet form $(\sE,\sF)$ on $L^2(E,\mu)$ is associated to an SFH $X$ on $E$,  if and only if it is strictly irreducible and strongly local-like.  Meanwhile  $(\sE,\sF)$ admits the representation
\begin{equation}\label{eq:51-3}
\begin{aligned}
	\sF=& \{f\in L^2(E,\mu)\cap  H^1_{e,\bs}(E): f(j)=0\\
	 &\qquad \qquad \qquad \text{whenever }j\notin E\text{ and }|\bs(j)|<\infty\text{ for }j=l\text{ or }r\},  \\
	 \sE(f,g)=&\frac{1}{2}\int_E \frac{df}{d\bs}\frac{dg}{d\bs}d\bs+\frac{1}{2}\sum_{k\geq 1}\frac{(f(b_k)-f(a_k))(g(b_k)-g(a_k))}{\bs(b_k)-\bs(a_k)},\quad f,g\in \sF,
\end{aligned}
\end{equation}
where $\bs$ is the scale function of $X$ taking $\mu$ to be the speed measure.  
\end{theorem}
\begin{remark}
At a heuristic level,  the strict irreducibility of $(\sE,\sF)$ corresponds to (SR) of $X$,  and strongly local-like property corresponds to (SF) and (SK).  The condition $\text{Cap}(\{l\})>0$ for $l\in E$  (or $\text{Cap}(\{r\})>0$ for $r\in E$)  in the definition of strictly irreducibility is not essential to our discussions.  Without assuming it,  $(\sE,\sF)$ still corresponds to another Markov process very similar to SFH,  except for that $l$ may be not arrived at by the process.  This is how a $d$-Bessel process on $[0,\infty)$ ($d\geq 2$) acts at the left endpoint $0$.  More examples are referred to in,  e.g.,  \cite[Example~3.11]{L23}.  
\end{remark}

\subsection{Proof of necessity}

We first complete the easy part of the proof of Theorem~\ref{THM53}.  In this proof of necessity,   the representation \eqref{eq:51-3} of the Dirichlet form associated to an SFH is also obtained.   

\begin{proof}[Proof of necessity]
We first let $X$ be an SFH on $E$ associated to a Dirichlet form $(\sE,\sF)$ on $L^2(E,\mu)$.  Since the symmetrizing measure of $X$ is unique (up to multiplicative constant),  one may take $\mu$ to be its speed measure.  The corresponding scale function is denoted by $\bs$.  
Let $\widehat{X}$ be the time-changed Brownian motion on $\bs(E)$ with speed measure $\tilde{\mu}:=\mu\circ \bs^{-1}$ appearing in \S\ref{SEC41}.  The expression of its Dirichlet form $(\widehat{\sE},\widehat{\sF})$ is obtained in \cite[Theorems~3.4 and 3.6]{L23}.  Using this expression and noting Remark~\ref{RM42},  one  can easily obtain the representation \eqref{eq:51-3} of $(\sE,\sF)$.  Since $(\widehat{\sE},\widehat{\sF})$ is irreducible and each singleton is of positive capacity with respect to $\widehat{\sE}$ (see \cite[Theorem~3.4 and Corollary~3.7]{L23}),  it follows that analogical properties hold for $(\sE,\sF)$.  Particularly $(\sE,\sF)$ is strictly irreducible.  In addition,  in view of the representation \eqref{eq:51-3},  one can verify the strongly local-like property of $(\sE,\sF)$ straightforwardly.
\end{proof}

\subsection{Proof of sufficiency}

To the contrary,  let $(\sE,\sF)$ be a regular,  strictly irreducible and strongly local-like Dirichlet form on $L^2(E,\mu)$.  Every function $f\in \sF$ takes its $\sE$-quasi-continuous version tacitly.  Write the Beurling-Deny decomposition of $(\sE,\sF)$ as follows:
\[
\sE(f,f)=\sE^{(c)}(f,f)+\int_{E\times E\setminus \mathbf{d}}(f(x)-f(y))^2J(dxdy)+\int_E f(x)^2k(dx),\quad f\in \sF,
\]
where $\sE^{(c)}$ is the strongly local part,  $\mathbf{d}$ is the diagonal of $E\times E$, $J$ is the jumping measure and $k$ is the killing measure.  The lemme below states that the strongly local-like property leads to the vanishing of killing part.   Particularly,  (SK) holds for the associated Hunt process of $(\sE,\sF)$.

\begin{lemma}\label{LM55}
 $k=0$.
\end{lemma}
\begin{proof}
Fix $u\in \sF\cap C_c(E)$ with $u\geq 0$.  Put $c:=\inf\{x: x\in \text{supp}[u]\}$ and $d:=\sup\{x:x\in \text{supp}[u]\}$.  Clearly $c,d\in E$.  Take $E\ni c'\leq c$ such that $c'<c$ if $c\neq l$ and  $c'<a_k$ if $c=b_k$ and $a_k>l$ for some $k$.  Analogously take $E\ni d'\geq d$ such that $d'>d$ if $d\neq r$ and $d'>b_k$ if $d=a_k$ and $b_k<r$ for some $k$.  Using the regularity of $(\sE,\sF)$  one may find a function $v\in \sF\cap C_c(E)$ such that $0\leq v\leq 1$ and $v\equiv 1$ on $[c',d']\cap E$.  It is straightforward to verify that $v$ is constant on a strong neighbour of $\text{supp}[u]$.  Hence $\sE(u,v)=0$.  Since $\sE^{(c)}(u,v)=0$,  it follows that
\[
\begin{aligned}
	0&=\int_{(K\times K^c) \cup (K^c\times K)}(u(x)-u(y))(v(x)-v(y))J(dxdy)+\int_E u(x)k(dx) \\
	&=2\int_{K\times K^c}u(x)(1-v(y))J(dxdy)+\int_E u(x)k(dx),
\end{aligned}\]
where $K:=\{x\in E: v(x)=1\}\supset [c',d']\cap E\supset \text{supp}[u]$.  Note that $u\geq 0$ and $0\leq v\leq 1$.  The above identity yields that $\int_E u(x)k(dx)=0$ for any $u\in \sF\cap C_c(E)$ with $u\geq 0$.  Eventually we can conclude $k=0$.  That completes the proof.  
\end{proof}

Another lemma below treats the jumping part of $(\sE,\sF)$.  To avoid ambiguity, $(x,y)_2$ for $x,y\in E$ with $x\neq y$ stands for the point in $E\times E\setminus \mathbf{d}$.  %For $x\in E$,  define $\rho(x):=\inf\{a_k: a_k>x,  k\geq 1\}$ and $\iota(x):=\sup\{b_k: b_k<x, k\geq 1\}$.  

\begin{lemma}\label{LM56}
$F:=\{(a_k,b_k)_2,  (b_k,a_k)_2: k\geq 1\}$ is a closed subset of $E\times E \setminus \mathbf{d}$ and $J$ is supported on $F$,  i.e.  $J(G)=0$ for $G:=\left(E\times E\setminus \mathbf{d}\right)\setminus F$.  
\end{lemma}
\begin{proof}
For simplification we still use $(x,y)$ to represent $(x,y)_2$ in this proof.  
We first prove that $F_1:=\{(a_k,b_k):k\geq 1\}$ is closed in $E\times E \setminus \mathbf{d}$.  Analogously $F_2:=\{(b_k,a_k):k\geq 1\}$ is also closed and hence $F=F_1\cup F_2$ is closed.  To do this,  argue by contradiction and suppose $(a_k,b_k)\rightarrow (x,y)\notin F_1$ for $x,y\in E$ with $x\neq  y$ as $k\rightarrow \infty$.  Particularly,  $a_k\rightarrow x,  b_k\rightarrow y$ and on account of $a_k<b_k$,  one gets $x<y$.  Taking a subsequence if necessary,  one may assume that $a_k$ is decreasing.  Then $a_k\geq b_{k+1}$.  As a result,  $x=\lim_{k\rightarrow \infty}a_k\geq \lim_{k\rightarrow \infty} b_{k+1}=y$,  as leads to a contradiction with $x<y$.  

Now put
\[
\begin{aligned}
	\mathscr C&:=\bigg\{\sum_{i=1}^p u_i(x)v_i(y): u_i,v_i\in \sF\cap C_c(E), \text{supp}[u_i]\cap \text{supp}[v_i]=\emptyset,  \\
	&\qquad\qquad \qquad u_i(a_k)v_i(b_k)=u_i(b_k)v_i(a_k)=0,\forall  k\geq 1,  1\leq i\leq p, p\geq 1\bigg\}. 
\end{aligned}\]
Clearly $\mathscr C\subset C_c(G)$ is an algebra.  Let us verify that $\mathscr C$ separates points of $G$.  To accomplish this fix $(x_1,y_1), (x_2,y_2)\in G$ with $x_1< x_2$.  We only consider the case $x_1<y_1$,  and the other case $x_1>y_1$ can be treated similarly.  When there are no intervals in \eqref{eq:50} between $x_1$ and $x_2\wedge y_1$,  we can take $\gamma,  \delta$ with $x_1<\gamma<\delta<x_2\wedge y_1$ and $u,v\in \sF\cap C_c(E)$ such that
\[
\begin{aligned}
	&u(x_1)=1,\quad u(z)=0,\;\forall z> \gamma,  \\
	&v(y_1)=1,\quad v(z)=0, \forall z<\delta.  
\end{aligned}
\]
Then $(x,y)\mapsto u(x)v(y)\in \mathscr C$ with $u(x_1)v(y_1)=1$ and $u(x_2)v(y_2)=0$.  When $x_1\leq a_k<b_k\leq x_2\wedge y_1$ for some $k$,  we have $y_1>b_k$ once $x_1=a_k$ because of $(x_1,y_1)\notin F$.  Hence one can take $u,v\in \sF\cap C_c(E)$ such that
\[
\begin{aligned}
&u(x_1)=1,\; u(z)=0,\;\forall z\geq b_k, \text{ and }u(a_k)=0\text{ whenever }x_1<a_k;  \\
	&v(y_1)=1, \; v(z)=0, \forall z\leq a_k, \text{ and }v(b_k)=0\text{ whenever }x_1=a_k.  
\end{aligned}
\]
It still holds that $(x,y)\mapsto u(x)v(y)\in \mathscr C$ with $u(x_1)v(y_1)=1$ and $u(x_2)v(y_2)=0$.  Eventually we can obtain that $\mathscr C$ separates points of $G$.  By the above argument one also finds that for any $(x,y)\in G$,  there exists $f\in \mathscr C$ such that $f(x,y)\neq 0$.  Therefore the Stone-Weierstrass theorem yields that $\mathscr C$ is dense in $C_c(G)$.  

Take $u,v\in \sF\cap C_c(E)$ with $\text{supp}[u]\cap \text{supp}[v]=\emptyset$ and $u(a_k)v(b_k)=u(b_k)v(a_k)=0$ for any $k\geq 1$.  The strongly local-like property implies that 
\[
	0=\sE(u,v)=-2\int_{E\times E\setminus \mathbf{d}} u(x)v(y)J(dxdy).  
\]
Then a standard argument yields that $\int_{E\times E\setminus \mathbf{d}} f(x,y)J(dxdy)=0$ for any $f\in C_c(G)$.  Therefore $J(G)=0$.  That completes the proof.  
\end{proof}

Denote by $X$ the associated Hunt process of $(\sE,\sF)$.  We are to prove that $X$ is an SFH.  

\begin{proof}[Proof of sufficiency]
We first prove that (SF) holds in the sense of $\sE$-q.e.,  i.e.  for $\sE$-q.e.  $x\in E$,  $\mathbf{P}_x$-a.s.,  
\begin{equation}\label{eq:53-2}
	(X_{t-}\wedge X_t,  X_{t-}\vee X_t)\cap E=\emptyset,\quad \forall t>0.  
\end{equation}
Clearly \eqref{eq:53-2} amounts to that the interval $(X_{t-},X_t)$ must be $(a_k,b_k)$ or $(b_k,a_k)$ for some $k$.  To accomplish this,  denote by $(N,H)$ the L\'evy system of $X$ (see,  e.g.,  \cite[Definition~A.3.7]{FOT11}).  Take arbitrary non-negative Borel measurable function $f$ on $E\times E$ with $f|_{\mathbf{d}\cup F}=0$,  where $F$ appears in Lemma~\ref{LM56}.  We have for any $x\in E$ and $T>0$ (recall that $k=0$),
\[
	\mathbf{E}_x\sum_{0<t\leq T} f(X_{t-},X_t)=\mathbf{E}_x\int_0^T Nf(X_t)dH_t,
\]
where $Nf(X_t)=\int_E f(X_t, y)dH_t$.  Denote by $\alpha_T(x)$ the right hand side of this identity.  It follows from \cite[Theorem~5.1.3]{FOT11} that for any non-negative Borel function $h$ on $E$,  
\[
	\int_E \alpha_T(x)h(x)\mu(dx)=\int_0^T \langle Nf \cdot \nu_H,  P_th\rangle dt,
\]
where $\nu_H$ is the Revuz measure of $H$ and $P_t$ is the semigroup of $(\sE,\sF)$.  Lemma~\ref{LM56} yields
\[
	Nf(x)\nu_H(dx)=\int_{y\in E} f(x,y)N(x,dy)\nu_H(dx)=\int_{y\in E} f(x,y)J(dxdy)\equiv 0.  
\]
Hence
\[
\int_E \lim_{T\uparrow \infty}\alpha_T(x)h(x)\mu(dx)=\lim_{T\uparrow \infty} \int_E \alpha_T(x)h(x)\mu(dx)=0.  
\]
Particularly,  
\begin{equation}\label{eq:54-2}
	\mathbf{E}_x \int_0^\infty Nf(X_t)dH_t=0,\quad \mu\text{-a.e.  } x\in E.  
\end{equation}
One may easily verify that \eqref{eq:54-2} amounts to 
\[
	\mathbf{E}_x \int_0^\infty e^{-t} Nf(X_t)dH_t=0,\quad \mu\text{-a.e.  } x\in E.  
\]
In view of \cite[Lemma~1.5.3]{FOT11},  $x\mapsto \mathbf{E}_x \int_0^\infty e^{-t} Nf(X_t)dH_t$ is $\sE$-quasi-continuous.  As a result,  
\[
\mathbf{E}_x\sum_{t>0} f(X_{t-},X_t)=\mathbf{E}_x\int_0^\infty Nf(X_t)dH_t=0
\]
for $\sE$-q.e. $x\in E$.  This readily arrives at \eqref{eq:53-2}.  

Next we prove that $\text{Cap}(\{x\})>0$ for any $x\in E$.  Due to the strictly irreducibility we already know that $l$ or $r$ is of positive capacity whenever $l\in E$ or $r\in E$.  Argue by contradiction and suppose $\text{Cap}(\{x\})=0$ for some $x\in (l,r)\cap E$.  Note that $\mu([l,x)\cap E),  \mu((x,r]\cap E)>0$.  Set $B:=(x,r]\cap E$.  On account of $\text{Cap}([l,x)\cap E)>0,  \text{Cap}(B)>0$,  \cite[Theorem~4.7.1]{FOT11} and \eqref{eq:53-2},  there exists $a\in [l,x)\cap E$ such that 
\[
	\mathbf{P}_a(T_B<\infty)>0,\quad \mathbf{P}_a(T_x<\infty)=0,
\]
where $T_B:=\inf\{t>0: X_t\in B\}$ and $T_x:=\{t>0: X_t=x\}$, 
and \eqref{eq:53-2} holds for $\mathbf{P}_a$-a.s.  Note that 
\[
	\widehat{T}_{x}=\inf\{t>0: X_{t-}=x\}\geq T_{x}=\infty,\quad \mathbf{P}_a\text{-a.s.};
\]
see,  e.g.,  \cite[Theorem~A.2.3]{FOT11}.  Hence for $\mathbf{P}_a$-a.s.,  $X_t\neq x$ and $X_{t-}\neq x$ for any $t\geq 0$.  For $\mathbf{P}_a$-a.s. $\omega$ with $t_\omega:=T_B(\omega)<\infty$,  it holds that $X_{t_\omega}(\omega)>x$ by the right continuity of $t\mapsto X_t(\omega)$ and $X_t(\omega)\neq x$ for any $t\geq 0$.  Similarly $X_{t_\omega-}(\omega)<x$.  Consequently $x\in (X_{t_\omega-}(\omega), X_{t_\omega}(\omega))$ for $\mathbf{P}_a$-a.s.  $\omega$ with $T_B(\omega)<\infty$.  This leads to a contradiction with \eqref{eq:53} for $\mathbf{P}_a$-a.s.  $\omega$.  

The arguments in the above two steps lead to (SF) for $X$.  The property (SR) for $X$ is a consequence of $\text{Cap}(\{x\})>0$ for any $x\in E$ and \cite[Theorem~4.7.1]{FOT11},  and (SK) is implied by Lemma~\ref{LM55}.  Eventually we can conclude that $X$ is an SFH.  That completes the proof. 
\end{proof}

\section*{Acknowledgment}

The idea of the proof of Lemma~\ref{LM54}~(6) was inspired by the discussion with a PhD student Dongjian Qian.  The author would like to thank him for this helpful discussion. 

%Here are further related references: \cite{LM20,  KW82,  S79,  K75,  W74}. 

\bibliographystyle{siam} % We choose the "plain" reference style
\bibliography{ExtQua} % Entries are in the "refs.bib" file

\begin{thebibliography}{10}

\bibitem{BG68}
{\sc R.~M. Blumenthal and R.~Getoor}, {\em {Markov processes and potential
  theory}}, Pure and Applied Mathematics, Vol. 29, Academic Press, New
  York-London, 1968.

\bibitem{BK87}
{\sc G.~Burkhardt and U.~K\"uchler}, {\em {The semimartingale decomposition of
  one-dimensional quasidiffusions with natural scale}}, Stochastic Process.
  Appl., 25 (1987), pp.~237--244.

\bibitem{CF12}
{\sc Z.-Q. Chen and M.~Fukushima}, {\em {Symmetric Markov processes, time
  change, and boundary theory}}, vol.~35 of London Mathematical Society
  Monographs Series, Princeton University Press, Princeton, NJ, 2012.

\bibitem{CFY06}
{\sc Z.-Q. Chen, M.~Fukushima, and J.~Ying}, {\em {Traces of symmetric Markov
  processes and their characterizations}}, Ann. Probab., 34 (2006),
  pp.~1052--1102.

\bibitem{DM82}
{\sc C.~Dellacherie and P.-A. Meyer}, {\em {Probabilities and potential. B}},
  vol.~72 of North-Holland Mathematics Studies, North-Holland Publishing Co.,
  Amsterdam, 1982.

\bibitem{F71-2}
{\sc M.~Fukushima}, {\em {Dirichlet spaces and strong Markov processes}},
  Trans. Amer. Math. Soc., 162 (1971), pp.~185--224.

\bibitem{F71}
\leavevmode\vrule height 2pt depth -1.6pt width 23pt, {\em {Regular
  representations of Dirichlet spaces}}, Trans. Amer. Math. Soc., 155 (1971),
  pp.~455--473.

\bibitem{FHY04}
{\sc M.~Fukushima, P.~He, and J.~Ying}, {\em {Time changes of symmetric
  diffusions and Feller measures}}, Ann. Probab., 32 (2004), pp.~3138--3166.

\bibitem{FOT11}
{\sc M.~Fukushima, Y.~Oshima, and M.~Takeda}, {\em {Dirichlet forms and
  symmetric Markov processes}}, vol.~19 of de Gruyter Studies in Mathematics,
  Walter de Gruyter {\&} Co., Berlin, extended~ed., 2011.

\bibitem{G75}
{\sc J.~Groh}, {\em {\"Uber eine Klasse eindimensionaler Markovprozesse}},
  Math. Nachr., 65 (1975), pp.~125--136.

\bibitem{HWY92}
{\sc S.~W. He, J.~G. Wang, and J.~A. Yan}, {\em {Semimartingale theory and
  stochastic calculus}}, Kexue Chubanshe (Science Press), Beijing; CRC Press,
  Boca Raton, FL, 1992.

\bibitem{I06}
{\sc K.~It{\^o}}, {\em {Essentials of stochastic processes}}, vol.~231 of
  Translations of Mathematical Monographs, American Mathematical Society,
  Providence, RI, 2006.

\bibitem{IM74}
{\sc K.~It{\^o} and H.~P. McKean~Jr}, {\em {Diffusion processes and their
  sample paths}}, Springer-Verlag, Berlin-New York, 1974.

\bibitem{KK74}
{\sc I.~S. Kac and M.~G. Krein}, {\em On the spectral functions of the string},
  vol.~103, American Mathematical Society, Providence, Rhode Island, 1974,
  p.~19–102.

\bibitem{KW09}
{\sc U.~Kant, T.~Klauss, J.~Voigt, and M.~Weber}, {\em Dirichlet forms for
  singular one-dimensional operators and on graphs}, J. Evol. Equ., 9 (2009),
  p.~637–659.

\bibitem{K75}
{\sc Y.~Kasahara}, {\em {Spectral theory of generalized second order
  differential operators and its applications to Markov processes}}, Japan. J.
  Math. (N.S.), 1 (1975), pp.~67--84.

\bibitem{K81}
{\sc F.~B. Knight}, {\em {Characterization of the L\'evy measures of inverse
  local times of gap diffusion}}, in Seminar on Stochastic Processes, 1981
  (Evanston, Ill., 1981), Birkh\"auser, Boston, Mass., 1981, pp.~53--78.

\bibitem{KW82}
{\sc S.~Kotani and S.~Watanabe}, {\em {Krein's spectral theory of strings and
  generalized diffusion processes}}, in Functional analysis in Markov processes
  (Katata/Kyoto, 1981), Springer, Berlin-New York, 1982, pp.~235--259.

\bibitem{K80}
{\sc U.~K\"uchler}, {\em Some asymptotic properties of the transition densities
  of one-dimensional quasidiffusions}, Publications of the Research Institute
  for Mathematical Sciences, 16 (1980), p.~245–268.

\bibitem{K86}
\leavevmode\vrule height 2pt depth -1.6pt width 23pt, {\em {On sojourn times,
  excursions and spectral measures connected with quasidiffusions}}, J. Math.
  Kyoto Univ., 26 (1986), pp.~403--421.

\bibitem{K87}
\leavevmode\vrule height 2pt depth -1.6pt width 23pt, {\em On Ito’s excursion
  law, local times and spectral measures for quasidiffusion</sup>s},
  Probability theory and mathematical statistics, Vol. II (Vilnius, 1985), VNU
  Sci. Press, Utrecht, 1987, p.~161–165.

\bibitem{L23}
{\sc L.~Li}, {\em On diffusions with discontinuous scales}, arXiv.org, math.PR
  (2022).

\bibitem{LM20}
{\sc Y.~Li and Y.~H. Mao}, {\em {Construction of generalized diffusion
  processes: the resolvent approach}}, Acta Math. Sin. (Engl. Ser.), 36 (2020),
  pp.~691--710.

\bibitem{M68}
{\sc P.~Mandl}, {\em {Analytical treatment of one-dimensional Markov
  processes}}, Die Grundlehren der mathematischen Wissenschaften, Band 151,
  Academia Publishing House of the Czechoslovak Academy of Sciences, Prague;
  Springer-Verlag New York Inc., New York, 1968.

\bibitem{M78}
{\sc I.~Monroe}, {\em {Processes that can be embedded in Brownian motion}},
  Ann. Probab., 6 (1978), pp.~42--56.

\bibitem{RY99}
{\sc D.~Revuz and M.~Yor}, {\em {Continuous martingales and Brownian motion}},
  vol.~293 of Grundlehren der Mathematischen Wissenschaften [Fundamental
  Principles of Mathematical Sciences], Springer-Verlag, Berlin, Berlin,
  Heidelberg, third~ed., 1999.

\bibitem{RW87}
{\sc L.~C.~G. Rogers and D.~Williams}, {\em {Diffusions, Markov processes, and
  martingales. Vol. 2}}, Wiley Series in Probability and Mathematical
  Statistics: Probability and Mathematical Statistics, John Wiley {\&} Sons,
  Inc., New York, 1987.

\bibitem{W74}
{\sc S.~Watanabe}, {\em {On time inversion of one-dimensional diffusion
  processes}}, Z. Wahrsch. Verw. Gebiete, 31 (1974), pp.~115--124.

\bibitem{YZ10}
{\sc J.~Ying and M.~Zhao}, {\em The uniqueness of symmetrizing measure of
  markov processes}, Proc. Amer. Math. Soc., 138 (2010), p.~2181–2185.

\end{thebibliography}

\end{document}